\documentclass[11pt]{article}
\usepackage{definitions}
    
\begin{document}

\title{On the Geometry and Refined Rate of Primal-Dual Hybrid Gradient for Linear Programming}
\author{Haihao Lu\thanks{The University of Chicago, Booth School of Business (haihao.lu@chicagobooth.edu).} \and Jinwen Yang\thanks{The University of Chicago, Department of Statistics (jinweny@uchicago.edu).}}
\date{December 2023}

\maketitle

\begin{abstract}
    We study the convergence behaviors of primal-dual hybrid gradient (PDHG) for solving linear programming (LP). PDHG is the base algorithm of a new general-purpose first-order method LP solver, PDLP, which aims to scale up LP by taking advantage of modern computing architectures. Despite its numerical success, the theoretical understanding of PDHG for LP is still very limited; the previous complexity result relies on the global Hoffman constant of the KKT system, which is known to be very loose and uninformative. In this work, we aim to develop a fundamental understanding of the convergence behaviors of PDHG for LP and to develop a refined complexity rate that does not rely on the global Hoffman constant. We show that there are two major stages of PDHG for LP: in Stage I, PDHG identifies active variables and the length of the first stage is driven by a certain quantity which measures how close the non-degeneracy part of the LP instance is to degeneracy; in Stage II, PDHG effectively solves a homogeneous linear inequality system, and the complexity of the second stage is driven by a well-behaved local sharpness constant of the system. This finding is closely related to the concept of partial smoothness in non-smooth optimization, and it is the first complexity result of finite time identification without the non-degeneracy assumption. An interesting implication of our results is that degeneracy itself does not slow down the convergence of PDHG for LP, but near-degeneracy does.

\end{abstract}


\section{Introduction}\label{sec:intro}


Linear programming (LP) is one of the most fundamental and important classes of optimization problems in operation research and computer science with a vast range of applications, such as network flow, revenue management, transportation, scheduling, packing and covering, and many others \cite{boyd2004convex,dantzig2002linear,anderson2000hotel,bowman1956production,charnes1954stepping,hanssmann1960linear,liu2008choice,manne1960linear}. Since the 1940s, LP has been extensively studied in both academia and industry. The state-of-the-art methods to solve LP problems, simplex methods \cite{dantzig1998linear} and interior-point methods (IPMs) \cite{karmarkar1984new}, are quite reliable to provide solutions with high accuracy and they serve as the base algorithms for nowadays commercial LP solvers. The success of both methods depends heavily on the efficient factorization methods to solve linear systems arising in the updates, which makes both algorithms highly challenging to further scale up. There are two fundamental reasons for this: (1) the storage of the factorization is quite memory-demanding and it usually requires significantly more memory than storing the original LP instance; (2) it is highly challenging to take advantage of modern computing architectures such as distributed computing and high-performance computing on GPUs when solving linear systems because factorization is sequential in nature.

Recent applications surge the interest in developing new algorithms for LPs with scale far beyond the capability of simplex methods and IPMs. {Matrix-free, i.e., no need to solve linear system, is a central feature of promising candidate algorithms, which guarantees low per-iteration computational cost and being friendly to distributed computation.} In this sense, first-order methods (FOMs) have become an attractive solution. The update of FOMs requires only the gradient information and is known for its capability of parallelization, thanks to the recent development in deep learning. In the context of LP, the computational bottleneck of FOMs is merely matrix-vector multiplication which is fairly cheap and exhibits easily distributed implementation, in contrast to solving linear systems for simplex methods and IPMs that is costly and highly nontrivial to parallelize. 

One notable instance that exemplifies the effectiveness of such methodology is the recent development of an open-source LP solver PDLP\footnote{The solver is open-sourced at \href{https://developers.google.com/optimization}{Google OR-Tools}.}~\cite{applegate2021practical}. A distributed version of PDLP has been used to solve real-world LP instances with as many as 92 billion non-zeros in the constraint matrix, which is a hundred to a thousand times larger than the scale state-of-the-art commercial LP solvers can solve~\cite{blog}. 
The base algorithm of PDLP is primal-dual hybrid gradient (PDHG) \cite{chambolle2011first}, a form of operator splitting method with alternating updates between the primal and dual variables. The implementation of PDLP also involves a few other enhancements/heuristics on top of PDHG, such as restart, preconditioning, adaptive step-size, etc, to further improve the numerical performance~\cite{applegate2021practical}. A significant difference between PDLP and other commercial LP solvers is that PDLP is totally matrix-free, namely, there is no necessity for PDLP to solve any linear system. The main computational bottleneck is matrix-vector multiplication which is much cheaper for solving large instances and more tailored for large-scale application in the distributed setting.

\begin{figure}[ht]
	\centering
	\hspace{-1Cm}
	\includegraphics[width=0.5\textwidth]{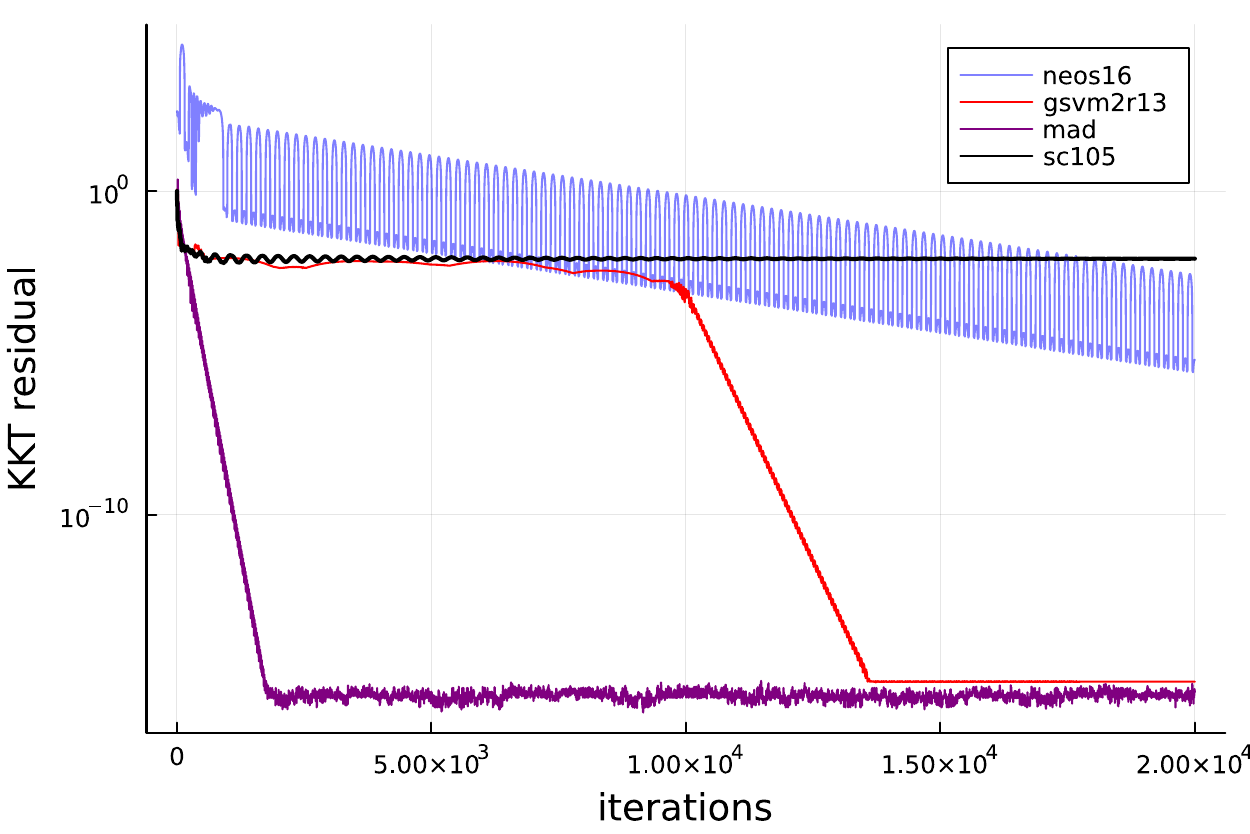}
	\caption{PDHG on four LP relaxation of instances from \texttt{MIPLIB}}
	\label{fig:pdhg}
\end{figure}

Despite the numerical success, there is a huge gap between the theoretical understanding and the practice of PDHG for LP. Recent works~\cite{lu2022infimal,applegate2023faster} present the complexity theory of (restarted) PDHG for LP, which shows that the iterates of PDHG linearly converge to an optimal solution, but the linear convergence rate depends on the global Hoffman constant of the KKT system, which is known to be exponentially loose. On the other hand, it is evident that the numerical performance of the algorithm does not depend on the overly-conservative Hoffman constant; instead, the algorithm often exhibits a two-stage behavior: an initial sublinear convergence followed by a linear convergence. For instances, Figure \ref{fig:pdhg} shows the representative behaviors of PDHG on four instances from \href{https://miplib.zib.de/tag_collection.html}{\texttt{MIPLIB 2017}}. The convergence patterns of the algorithm differ dramatically across these LP instances. The instance \texttt{mad} converges linearly to optimality within only few thousands of steps. Instances \texttt{neos16} and \texttt{gsvm2r13} eventually reach the linear convergence stage but beforehand there exists a relatively flat slow stage. The instance \texttt{sc105} does not exhibit linear rate within twenty thousand iterations. Furthermore, the eventual linear rates are not equivalent: \texttt{gsvm2r13} exhibits much faster linear rate than \texttt{neos16}. Similar observation holds for the ``warm-up" sublinear stage: \texttt{neos16} has much shorter sublinear period than \texttt{gsvm2r13}. The global linear convergence with the conservative Hoffman constant~\cite{lu2022infimal,applegate2023faster} is clearly not enough to interpret the diverse empirical behaviors. The goal of this paper is to bridge such gap by answering the following two questions:
\begin{itemize}
    \item How to understand the two-stage behaviors of PDHG for LP, and what are the geometric quantities that drive the length of the first stage and the convergence rate of the second stage?
    \item Can we obtain complexity theory of PDHG for LP without using the overly-conservative global Hoffman constant?
\end{itemize}

More formally, we consider standard-form LP:
\begin{align}\label{eq:standardform}
    \begin{split}
        \min_{x\geq 0}\ \  & c^T x \\
        s.t.\ \  & Ax=b \ ,
    \end{split}
\end{align}
and its primal-dual form:
\begin{align}\label{eq:minmax}
    \begin{split}
        \min_{x\geq 0}\max_{y}\; c^Tx-y^TAx+b^Ty \ .
    \end{split}
\end{align}

\begin{algorithm}
    \renewcommand{\algorithmicrequire}{\textbf{Input:}}
    \renewcommand{\algorithmicensure}{\textbf{Output:}}
    \caption{Primal Dual Hybrid Gradient (PDHG) for \eqref{eq:minmax}}
    \label{alg:pdhg}
    \begin{algorithmic}[1]
        \REQUIRE Initial point $z^0=(x^0,y^0)$, step-size $s>0$.
        \FOR{$k=0,1,...$}
        \STATE       
        $x^{k+1}=\mathrm{proj}_{\mathbb R_+^n}\pran{x^k-s(c-A^Ty^k)}$
        \STATE $y^{k+1}=y^k-s (A(2x^{k+1}-x^{k})-b)$
        \ENDFOR
    \end{algorithmic}
\end{algorithm}

Algorithm \ref{alg:pdhg} presents the primal-dual hybrid gradient method (PDHG, a.k.a Chambolle and Pock algorithm~\cite{chambolle2011first}) for LP \eqref{eq:minmax}. Without loss of generality, we assume the primal and the dual step-sizes are the same throughout the paper. This can be achieved by rescaling the problem instance~\cite{applegate2023faster}. The previous works~\cite{chambolle2016ergodic,fercoq2022quadratic,lu2022infimal} show the following complexity of PDHG for LP to achieve $\epsilon$-accuracy solution:
\begin{equation}\label{eq:existing-rate}
    \mathcal O\pran{\frac{\|A\|_2^2}{1/H^2}\log\pran{\frac{R}{\epsilon}}} \ ,
\end{equation}
where $R$ is essentially an upper bound on the norm of the iterates and $H$ is the Hoffman constant of the KKT system of LP:
\begin{equation*}
    Ax=b,\; x\geq 0,\; A^Ty\leq c,\; \frac{1}{R}(c^Tx-b^Ty)\leq 0 \ .
\end{equation*}
The major issue of the above complexity is the reliance on the global Hoffman constant $H$, which is notoriously conservative. Consider a general linear inequality system $Fx\leq g$. A commonly-used characterization to its Hoffman constant~\cite{pena2021new} is 
\begin{equation}\label{eq:hoffman-intro}
    H = \max_{\substack{J\in [m]\\A_J\;\text{full row rank}}}\frac{1}{\min_{\substack{v\in \mathbb R_+^{|J|}\\ \|v\|_2=1}}\|A_J^Tv\|_2} \ ,
\end{equation}
The inner part of the Hoffman constant is  an extension of the minimal non-zero singular value of the submatrix $F_J$, which is expected to appear in the linear convergence rate of first-order methods~\cite{nesterov2013introductory}. However, the outer maximization optimizes over exponentially many subsets of linear constraints, which makes the Hoffman constant overly conservative and cannot characterize the behaviors of the algorithms. Intuitively this is because when calculating the global Hoffman constant, one needs to look at the local geometry on every boundary set (i.e., extreme points, edges, faces, etc) of the feasible region, and consider the worst-case situation. See Appendix \ref{app:example} for an example where the Hoffman constant can be arbitrarily loose. 

In this paper, we show that the performance of PDHG for LP does not rely on the overly-conservative global Hoffman constant; instead, the algorithm exhibits a two-stage behavior:
\begin{itemize}
    \item In the first stage (see Section \ref{sec:stage-1} for more details), the algorithm aims to identify the non-degenerate variables. {Notice that the algorithm may not converge to an optimal solution that satisfies strict complementary slackness, which we call degeneracy in this paper, and this definition is consistent with the partial smoothness literature. We further call the primal variables that satisfy strict complementary slackness the non-degenerate variables.} This stage finishes within a finite number of iterations, and the convergence rate in this stage is sublinear. The driving force of the first stage is how the non-degenerate part of the LP is close to degeneracy.
    \item In the second stage (see Section \ref{sec:stage-2} for more details), the algorithm effectively solves a homogeneous linear inequality system. The algorithm converges linearly to an optimal solution, and the driving force of the linear convergence rate is a local sharpness constant for the homogeneous linear inequality system. This local sharpness constant is much better than the global Hoffman constant, and it is a generalization of the minimal non-zero singular value of a certain matrix. Intuitively, this happens due to the structure of the homogeneous linear inequality system so that one just needs to focus on the local geometry around the origin (See~\cite{pena2023easily} for a detailed characterization), avoiding going through the exponentially-many boundary set as in the calculation of global Hoffman constant.
\end{itemize}

To gain more intuitions of the two-stage convergence behavior, we consider a simple yet representative class of two-dimension house-shaped dual LPs parameterized by  $0\le \delta\le \kappa<1$ (see Figure \ref{fig:house-fig}):
\begin{align}\label{eq:house}
    \begin{split}
        & \ \max \; y_2 \\
        & \ \;\; \mathrm{s.t.}\;\; y_1\geq -1,\; y_1 \leq 1 \\
        & \ \quad\quad\;\; y_2\geq -1,\; y_2\leq \kappa-\delta \\ 
        & \ \quad\quad\;\; y_1+\frac{1}{\kappa} y_2 \leq 1,\; -y_1+\frac{1}{\kappa} y_2 \leq 1 \ .
    \end{split}
\end{align}
Parameter $\delta$ essentially measures how close the dual LP is to degeneracy. When $\delta=0$ (i.e., the top three constraints in Figure \ref{fig:house-fig} intersect), the problem is degenerate; when $\delta$ is close to $0$, the problem is near-degenerate; when $\delta$ is reasonably large, the problem is far away from degeneracy. Parameter $\kappa$ roughly measures the condition number of the constraint matrix, i.e., the ratio between the minimal and the maximal singular values of the matrix. Figure \ref{fig:house-result} presents the numerical performance of PDHG on this two-dimension instance with different choices of $\delta$ and $\kappa$. Indeed, with a proper choice of the parameters $\delta$ and $\kappa$, the convergence behavior of Figure \ref{fig:house-result} mimics well the real \texttt{MIPLIB} instances shown in Figure \ref{fig:pdhg}. Furthermore, one can see that the smaller the value of $\kappa$ is, the slower the eventual linear convergence is; the smaller the value of $\delta$ is, the longer the initial slow convergence stage. Yet, if $\delta=0$, i.e., the problem is degenerate, the problem does not have the slow first stage. The above numerical observations are not limited to this two-dimension example, and they are formalized for general LP in Section \ref{sec:stage-1} and Section \ref{sec:stage-2}.

\pgfmathsetmacro{\utau}{1}
\pgfmathsetmacro{\aa}{3}
\pgfmathsetmacro{\bb}{1/10}
\begin{figure}
\centering
\begin{subfigure}{0.4\textwidth}
  \centering
  \hspace{-1cm}
  \begin{tikzpicture}[scale=0.7,every node/.style={scale=1.5}]
        \begin{axis}[xlabel={},ylabel={}, ytick={0}, xtick={0}, yticklabels={}, xticklabels={}, extra x ticks={0, 1, 2},extra x tick labels={$-1$, $0$, $1$},  extra y ticks={2, 1.7, 1, 0}, extra y tick labels={$\kappa$, $\kappa-\delta$, $0$, $-1$},legend pos=outer north east, legend cell align=left, xmin=-0.25, xmax=2.25, ymin=-0.25, ymax=2.4, y=2.8cm,  x=2.8cm,]
        \filldraw [fill=gray!40] (0,0) -- (2,0) -- (2,1) -- (0,1) -- (0,0);
        \addplot[forget plot,color=black,line width=1.5pt, fill=gray!40] coordinates {(0,0) (0,1) (0.7,1.7) (1.3,1.7) (2,1) (2,0)};
        \addplot[forget plot,color=black,line width=1.5pt] coordinates {(0.7,1.7) (1,2) (1.3,1.7)};
        \addplot[forget plot,color=black,line width=1.5pt] coordinates {(0,1.7) (2,1.7)};
        \addplot[forget plot,color=black,line width=1.5pt] coordinates {(0,0) (2,0)};
        \addplot[-Stealth, color=black,line width=1.5pt] coordinates {(2.1,1.8) (2.1,2.3)};
        \end{axis}
        \end{tikzpicture}
  \caption{The LP instance \eqref{eq:house}}
  \label{fig:house-fig}
\end{subfigure}%
\begin{subfigure}{0.5\textwidth}
  \centering
  \includegraphics[width=\textwidth]{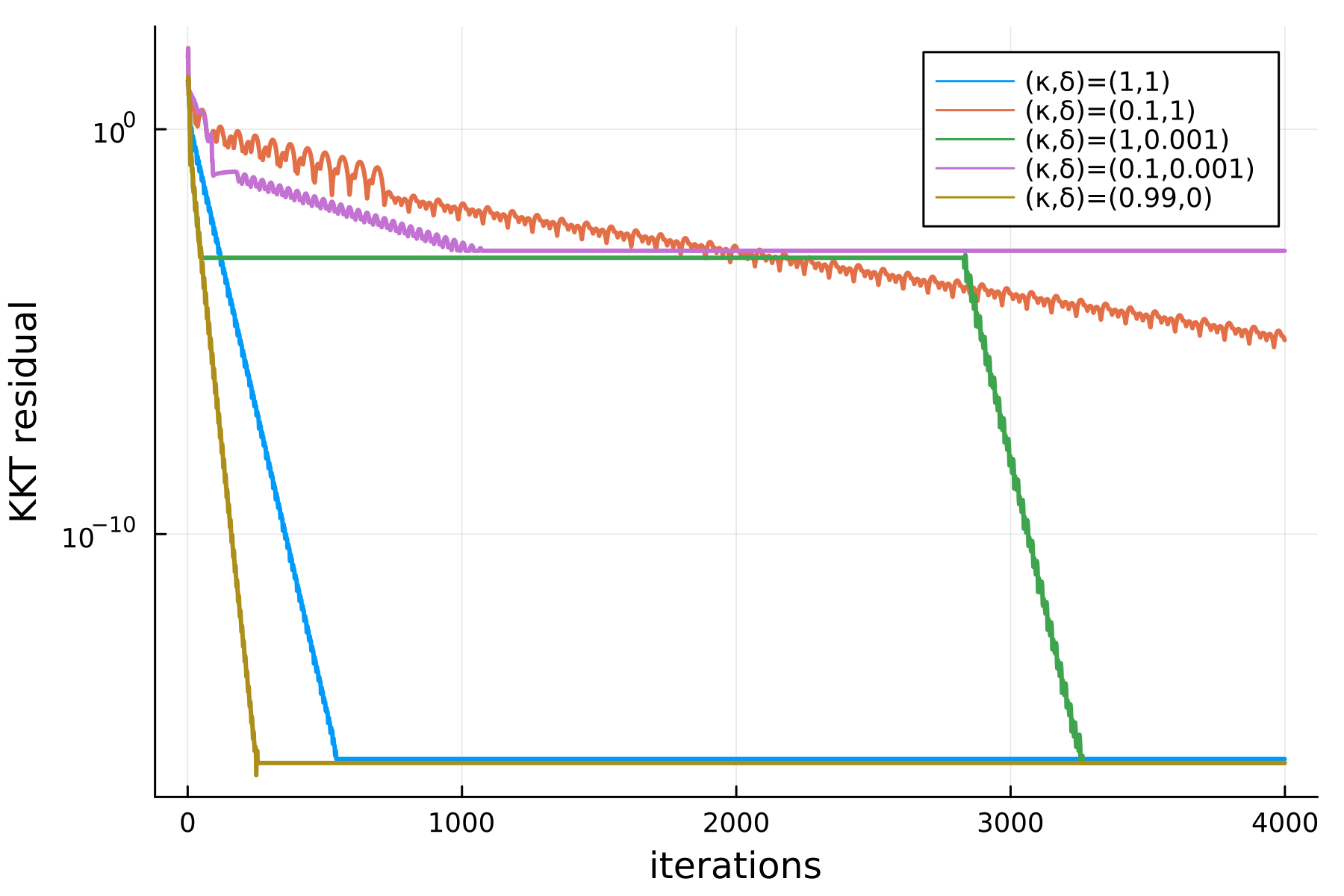}
  \caption{Performance of PDHG}
  \label{fig:house-result}
\end{subfigure}
\caption{Plots to illustrate the geometry of the LP instance \eqref{eq:house}, and the numerical behaviors of PDHG for solving the LP instance with different parameter $\kappa,\delta$.}
\label{fig:house}
\end{figure}
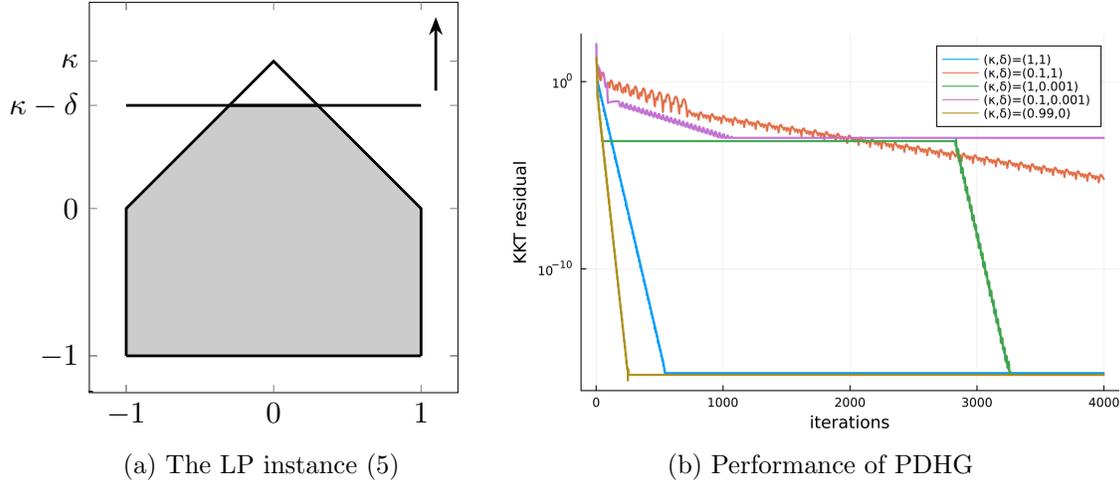

Indeed, this two-stage behavior has been documented in the literature of interior-point methods and non-smooth optimization. For example, \cite{wright1997primal,guler1993convergence,ye1992finite} discuss the two-stage behaviors of interior-point method (IPM) for LP: IPM has linear convergence in the first stage and super-linear convergence in the second stage. The phase transition happens when certain active basis are identified, and the length of the first phase depends on a certain metric that measures how close the analytical center of the optimal solution set (i.e., the converging optimal solution) is to degeneracy. In the context of non-smooth optimization, there are also fruitful results studying identification of first-order methods under partial smoothness~\cite{wright1993identifiable,lewis2016proximal,liang2017activity,liang2017local}. However, almost all of these works require the non-degeneracy condition, which, in the context of LP, refers to the algorithm converges to an optimal solution that satisfies strict complementary slackness. Unfortunately, this condition barely holds for real-world LP instances with PDHG as well as many other problems, and thus it is also called ``irrepresentable condition'' in the literature~\cite{fadili2019model,fadili2018sensitivity}. Note that the identification results of IPMs also have an implicit reliance on the non-degeneracy condition, since the converging optimal solution of IPMs  (i.e., analytical center of optimal solution face) always satisfies strict complementary slackness~\cite{wright1997primal,ye1992finite,guler1993convergence}. {In contrast to these literatures, our work presents the first complexity result on finite time identification without the non-degeneracy assumption, in the context of LP. Indeed, our results suggest that degeneracy itself does not slow down  the convergence  for first-order methods, but near-degeneracy does.}

\subsection{Contributions}
The goal of the paper is to provide a theoretical foundation for PDLP, a new first-order method LP solver based on PDHG. The contributions of the paper can be summarized as follows:
\begin{itemize}
    \item Motivated by the empirical behaviors of PDHG for LP, we propose a two-stage characterization of its convergence behaviors:
    \begin{itemize}
        \item In the first stage, PDHG attempts to identify the active variables and the near-to-degeneracy parameter drives the identification complexity.
        \item In the second stage, PDHG attempts to solve a homogeneous linear inequality system, and a local sharpness constant controls the linear convergence rate.
    \end{itemize}

    \item We provide the first linear convergence complexity of PDHG for LP without the dependency on the global Hoffman constant.

    \item Of independent interest, we present the first complexity bound on finite-time identification in the partial-smoothness context without the ``irrepresentative'' non-degeneracy condition.
\end{itemize}

\subsection{Related Literature}
{\bf FOM solvers for LP.} 
Recently, first-order methods become increasingly appealing for solving large LP due to their low per-iteration cost \cite{applegate2021infeasibility,applegate2023faster,basu2020eclipse,li2020asymptotically,lin2021admm,wang2017new,yen2015sparse}.
\begin{itemize}
    \item PDLP~\cite{applegate2021practical} utilizes primal-dual hybrid gradient method (PDHG) as its base algorithm and introduces practical algorithmic enhancements, such as presolving, preconditioning, adaptive restart, adaptive choice of step size, and primal weight, on top of PDHG. Right now, it has three implementations:  a prototype implemented in Julia (\href{https://github.com/google-research/FirstOrderLp.jl}{FirstOrderLp.jl}) for research purposes, a production-level C++ implementation that is included in Google \href{https://developers.google.com/optimization}{OR-Tools}, and an internal distributed version at Google. The internal distributed version of PDLP has been used to solve real-world problems with as many as 92 billion non-zeros~\cite{blog}, which is one of the largest LP instances that are solved by a general-purpose LP solver.
    
    \item \href{https://github.com/leavesgrp/ABIP}{ABIP}~\cite{lin2021admm,deng2022new} is an ADMM-based interior-point method. The core algorithm of ABIP is still a homogeneous self-dual embedded interior-point method. Instead of approximately minimizing the log-barrier penalty function with a Newton step, ABIP utilizes multiple steps of alternating direction method of multipliers (ADMM). The  $\mathcal O\left(\frac{1}{\epsilon}\log\left(\frac{1}{\epsilon}\right)\right)$ sublinear complexity of ABIP was presented in \cite{lin2021admm}. Recently, \cite{deng2022new} includes new enhancements, i.e., preconditioning, restart, hybrid parameter tuning, on top of ABIP (the enhanced version is called ABIP+). ABIP+ is numerically comparable to the Julia implementation of PDLP. ABIP+ now also supports a more general conic setting when the proximal problem associated with the log-barrier in ABIP can be efficiently computed.

    \item ECLIPSE~\cite{basu2020eclipse} is a distributed LP solver designed specifically for addressing large-scale LPs encountered in web applications. These LPs have a certain decomposition structure, and the effective constraints are usually much less than the number of variables. ECLIPSE looks at a certain dual formulation of the problem, then utilizes accelerated gradient descent to solve the smoothed dual problem with Nesterov's smoothing. This approach is shown to have $\mathcal O(\frac{1}{\epsilon})$ complexity, and it is used to solve web applications with $10^{12}$ decision variables~\cite{basu2020eclipse} and real-world web applications at LinkedIn platform~\cite{ramanath2022efficient,acharya2023promoting}.
    
    \item \href{https://github.com/cvxgrp/scs}{SCS}~\cite{o2016conic,o2021operator} tackles the homogeneous self-dual embedding of general conic programming using ADMM. As a special case of conic programming, SCS can also be used to solve LP. Each iteration of SCS involves projecting onto the cone and solving a system of linear equations with similar forms so that it only needs to store one factorization in memory. Furthermore, SCS supports solving linear equations with an iterative method, which only uses matrix-vector multiplications. 
\end{itemize}



{\bf Complexity theory of FOMs for LP.} 
One major drawback of applying first-order methods to solving LP is their slow tail convergence. Due to the lack of strong convexity, one can only obtain sublinear convergence rate when directly applying the classic results of FOMs on LP~\cite{nesterov2013introductory,beck2017first}. This suggests that FOMs cannot produce the high-accuracy solutions that are typically expected from LP solvers within a reasonable time limit. Luckily, it turns out that LP satisfies additional growth condition structures that can improve the convergence of FOMs from sublinear to linear. For example, \cite{eckstein1990alternating} showed that a variant of ADMM can have linear convergence for LP. It is shown in \cite{liang2017activity,liang2017local,liang2018local} that under non-degeneracy condition, many primal-dual algorithms have eventual linear convergence when solving LP, even though it may take a long time before reaching the linear local regime. Recently, \cite{applegate2023faster} introduced a sharpness condition that is satisfied by LP, and proposed a restarted scheme for a variety of primal-dual methods to solve LP. It turns out that the restarted variants of many FOMs, such as PDHG, extra-gradient method (EGM) and ADMM, can achieve global linear convergence, and such linear convergence is the optimal rate, namely, there is a worst-case LP instance such that no FOMs can achieve better than such linear rate. Later on, \cite{lu2022infimal} shows that PDHG (without restart) also has linear convergence for LP, and the linear convergence rate is slower than that in \cite{applegate2023faster}. However, the obtained linear convergence rates in \cite{applegate2023faster,lu2022infimal} depend on the global Hoffman constant of the KKT system of LP, which is known to be very loose and cannot characterize the behaviors of the algorithm. In this paper, we aim to develop a tighter complexity theory of PDHG for LP {without using global Hoffman constant}.

{\bf Primal-dual hybrid gradient method (PDHG).} PDHG was initially proposed for applications in image processing and computer vision \cite{chambolle2011first,condat2013primal,esser2010general,he2012convergence,zhu2008efficient}. The first convergence guarantee of PDHG was obtained in \cite{chambolle2011first}, which shows the $\mathcal O(1/k)$ sublinear ergodic convergence rate of PDHG for convex-concave primal-dual problems. Later on, simplified analyses for the $\mathcal O(1/k)$ sublinear ergodic rate of PDHG were presented in \cite{chambolle2016ergodic,lu2023unified}. More recently, it is shown that the last iterates of PDHG exhibit linear convergence under a mild regularity condition that is satisfied by many applications, including LP~\cite{fercoq2022quadratic,lu2022infimal}. Moreover, many variants of PDHG have been proposed, including adaptive version \cite{goldstein2015adaptive,malitsky2018first,pock2011diagonal,vladarean2021first} and stochastic version \cite{alacaoglu2022convergence,chambolle2018stochastic}. It is also shown that PDHG is equivalent to Douglas-Rachford Splitting up to a linear transformation \cite{liu2021acceleration,o2020equivalence}.

{\bf Identification and partial smoothness.} 
Active-set identification is a classic topic in constrained optimization and in non-smooth optimization, with both theoretical and computational importance.  Partial smoothness, a notion meaning smoothness along some directions and sharpness in other directions, plays a crucial role in the analysis of identification. Under the partial smoothness assumption, a refined analysis is proposed for analyzing the identification property of certain algorithms. More specifically, manifold identification of dual averaging is shown in \cite{lee2012manifold}. It is shown in \cite{liang2014local,liang2017activity} that the finite time identification of forward-backward-type methods. In \cite{liang2017local,liang2018local}, it is shown the active set identification and local convergence rate of many primal-dual algorithms including PDHG and ADMM. In \cite{poon2018local}, it shows that stochastic methods also share a similar identification property. Similar results are derived for Newton's method in \cite{lewis2021active}. {\blue The eventual sharp linear rate of Douglas-Rachford Splitting on basis pursuit is derived under non-degeneracy condition in \cite{demanet2016eventual}.}  In   \cite{davis2021subgradient}, the non-smooth dynamic of sub-gradient methods near active manifolds is studied. 

However, all the above works assume the non-degeneracy condition of the problem, namely, $0\in \mathrm{ri}(\mathcal F(x^*))$, where $\mathrm{ri}(\cdot)$ represents the relative interior of a set and $\mathcal F(x^*)$ denotes the sub-differential set at the converging optimal solution $x^*$.  The only exceptions that do not require the non-degeneracy condition are~\cite{fadili2018sensitivity,fadili2019model}, where they showed that the active set of the iterates can potentially be larger than the active set of optimal solutions because of degeneracy. Unfortunately, it is unclear how to check the non-degeneracy condition in general, thus it is also called ``irrepresentable condition''~\cite{fadili2018sensitivity,fadili2019model}. In the case of LP, this condition means that the converging optimal solution satisfies strict complimentary slackness, which unfortunately is almost never the case when using FOMs to solve real-world LP instances. To the best of our knowledge, this work is the first complexity result of finite-time identification without the non-degeneracy condition.

On a related note, the identification behaviors of interior-point methods (IPMs) for LP have been observed ~\cite{wright1997primal,ye1992finite,guler1993convergence}, where it was shown that IPM has super-linear convergence after identification. Indeed, these results implicitly utilize the fact that IPM always converges to the analytical center of the optimal face, which satisfies strict complementary slackness~\cite{wright1997primal,ye1992finite,guler1993convergence}.

\subsection{Organization}
We discuss the sharpness of homogeneous linear inequality system and collect necessary convergence results of PDHG in Section \ref{sec:pre}. The analysis of identification stage (Stage I) and local convergence (Stage II) are presented respectively in Section \ref{sec:stage-1} and \ref{sec:stage-2}. Section \ref{sec:numerical} illustrates and verifies the theoretical results through numerical experiments. We conclude the paper and propose several future research directions in Section \ref{sec:conclusion}.

\subsection{Notations}
For matrix $A\in \mathbb R^{m\times n}$, denote $A_j$ the $j$-th column of matrix $A$ and $A_J=(A_j)_{j\in J}\in \mathbb R^{m\times |J|}$, $J\subset [n]$. Denote $\|A\|_2$ the operator norm of a matrix $A$.
Denote $\|x\|_2$ the Euclidean norm for a vector $x$ and $\langle x, y\rangle$ for its associated inner product. For a positive definite matrix $M\succ 0$, let $\langle x, y\rangle_M=\langle x,M y\rangle$ and $\|x\|_M=\sqrt{x^TMx}$ the norm induced by the inner product. Let $\mathrm{dist}_M(z,\mathcal Z)$ represent the distance between point $z$ and set $\mathcal Z$ under the norm $\|\cdot\|_M$, that is, $\mathrm{dist}_M(z,\mathcal Z)=\min_{u\in\mathcal Z}\|u-z\|_M$. In particular, $\mathrm{dist}(z,\mathcal Z)=\min_{u\in\mathcal Z}\|u-z\|_2$. Denote $P_{\mathcal Z}(z):=\arg\min_{u\in\mathcal Z}\{\|u-z\|_2\}$ the projection onto set $\mathcal Z$ under $l_2$ norm. Denote the optimal solution set to \eqref{eq:minmax} as $\mathcal Z^*$. $f(x)=\mathcal O(g(x))$ means that for sufficiently large $x$, there exists constant $C$ such that $f(x)\leq Cg(x)$. Let $\iota_C(\cdot)$ be the indicator function of set $\mathcal C$. Let $z=(x,y)$, and denote $\mathcal F(z)=\mathcal F(x,y)=\begin{pmatrix}
        c-A^Ty+\partial \iota_{\mathbb R_+^n}(x) \\ b-Ax
    \end{pmatrix}$ the sub-differential of \eqref{eq:minmax}. Let $P_s=\begin{pmatrix}
          \frac{1}{s} I & -A^T \\ -A & \frac{1}{s} I
        \end{pmatrix}$ and
    $\mathbf{1}_n$ represents all 1's vector of length $n$. For set $A$ and $B$, let $A\backslash B:=\{x\;|\: x\in A, x\notin B\}$. Let $[n]:=\{1,2,...,n\}$.

\section{Preliminaries}\label{sec:pre}
We here present preliminary results on sharpness of primal-dual problems, convergence properties of PDHG and sharpness of homogeneous linear inequality system that we will use later.

\subsection{Sharpness of primal-dual problems}\label{sec:sharp-condition}
Consider a generic convex-concave primal-dual problem
\begin{equation}\label{eq:pd}
    \min_{x\in \mathcal{X}} \max_{y\in \mathcal{Y}} L(x, y) \ ,
\end{equation}
where $L(x,y)$ is a convex function in $x$ and a concave function in $y$. We use $z=(x,y)$ to represent the primal and the dual solution together.

Normalized duality gap was introduced in \cite{applegate2023faster} as a progress metric for primal-dual problems:
\begin{mydef}
    For a primal-dual problem~\eqref{eq:pd} and a solution $z=(x,y)$, the normalized duality gap with radius $r$ is defined as
    \begin{equation}\label{eq:ndg}
    \rho_r(z)=\max_{\hat z \in W_r(z)}\frac{L(x,\hat y)-L(\hat x,y)}{r} \ ,
\end{equation}
where $W_r(z)=\{ \hat z\in\mathcal Z \;|\; \Vert z-\hat z\Vert_2\leq r \}$ is a ball centered at $z$ with radius $r$ intersected with $\mathcal Z=\mathcal X \times \mathcal Y$.
\end{mydef}

The normalized duality gap $\rho_r(z)$ is a valid progress metric for LP, since it provides an upper bound on the KKT residual:
\begin{prop}[{\cite[Lemma 4]{applegate2023faster}}]\label{prop:prop-res-gap}
It holds for any $z=(x,y)$ with $x\geq 0$ and $\|z\|_2\le R$ that {\blue for any $r\in(0,R]$,}
\begin{align*}
    \rho_r(z)\geq \left\Vert \begin{pmatrix}
        Ax-b \\ [A^Ty-c]^+ \\ \frac 1R [c^Tx-b^Ty]^+
    \end{pmatrix}  \right\Vert_2 \ .
\end{align*}
\end{prop}

The sharpness of a primal-dual problem is defined in \cite{applegate2023faster}:
\begin{mydef}[{\cite[Definition 1]{applegate2023faster}}]
    We say a primal-dual problem is $\alpha$-sharp on the set $\mathcal S$ if $\rho_r(z)$ is $\alpha$-sharp on $\mathcal S$ for all $r$, i.e., it holds for all $z\in \mathcal S$ that
    \begin{equation}\label{eq:sharp}
        \alpha\mathrm{dist}(z,\mathcal Z^*)\leq \rho_r(z) \ .
    \end{equation}
\end{mydef}
The following proposition shows that the primal-dual form of LP is sharp.
\begin{prop}[{\cite[Lemma 5]{applegate2023faster}}]\label{prop:global-sharp}
    LP in the primal-dual form \eqref{eq:minmax} is a sharp problem on set $\{z\;|\;\|z\|_2\le R\}$, i.e.,
    the normalized duality gap $\rho_r(z)$ defined in \eqref{eq:ndg} satisfies
    \begin{equation*}
        \frac{1}{H}\mathrm{dist}(z,\mathcal Z^*) \leq \rho_r(z) \ ,
    \end{equation*}
    for any $z\in\{z\;|\;\|z\|_2\le R\}$ {\blue and $r\in(0,R]$}, where $H$ is the Hoffman constant of the KKT system of LP:
    \begin{equation*}
    Ax=b,\; A^Ty\leq c,\; x\geq 0,\; \frac 1R (c^Tx-b^Ty)\leq 0 \ .
    \end{equation*}
\end{prop}

\subsection{Convergence of PDHG iterates}\label{sec:pdhg-iterate}
In this subsection, we present some basic convergence properties of PDHG on LP in literature. Many of these convergence results hold for generic primal-dual problems, and we here state them in the context of LP for convenience.


The next lemma presents the non-expansiveness and convergence of PDHG on LP:
\begin{lem}[\cite{chambolle2016ergodic}]\label{lem:property}
Consider the iterations $\{z^k\}_{k=0}^\infty$ of PDHG on LP (Algorithm \ref{alg:pdhg}). Let $\mathcal Z^*$ be the optimal solution set to \eqref{eq:minmax}. Then it holds for any $k$ that

(a) (non-expansiveness) $\Vert z^{k+1}-z^* \Vert_{P_s} \leq \Vert z^k-z^*\Vert_{P_s} $ for any $z^*\in\mathcal Z^*$,\\
(b) (convergence) there exists $z^*\in\mathcal Z^*$ such that $z^*=\lim_{k\rightarrow\infty} z^k$, and\\
(c) (property of the optimal solution set) $\|z^k-\tilde z^*\|_{P_s}\geq \| z^*-\tilde z^*\|_{P_s}$ for any $\tilde z^*\in\mathcal Z^*$, where $z^*$ is the converging optimal solution defined in (b).
\end{lem}


The next theorem presents the sublinear and linear convergence rate of PDHG on LP. These results can be obtained utilizing a similar proof technique as in \cite{lu2022infimal}. We present the formal proof in Appendix \ref{app:proof_thm_1} for completeness.
\begin{thm}\label{thm:thm-linear-z}
Consider the iterations $\{z^k\}_{k=0}^\infty$ of PDHG on LP (Algorithm \ref{alg:pdhg}). Suppose the step-size $s\leq \frac{1}{2\Vert A\Vert_2}$. Then it holds for any $k\geq 0$ that

(a) (sublinear rate)   
\begin{equation}
    \|z^{k+1}-z^{k}\|_{P_s}\leq \mathrm{dist}_{P_s^{-1}}(0,\mathcal F(z^k))\leq \frac{\mathrm{dist}_{P_s}(z^0,\mathcal Z^*)}{\sqrt k}\ .
\end{equation}

(b) (linear rate) 
\begin{equation}
         \Vert z^k-z^* \Vert_{P_s}\leq 4\exp\pran{-\frac{k}{2\left\lceil 4e/(s\alpha)^2\right\rceil}} \mathrm{dist}_{P_s}(z^0,\mathcal Z^*) \ ,
    \end{equation}
where $\alpha$ is the sharpness constant of LP along iterates, i.e., $\alpha\mathrm{dist}(z^k,\mathcal Z^*)\leq \rho_r(z^k)$ for any $k\geq 0$.
\end{thm}

We comment that we assume the step-size $s\leq \frac{1}{2\|A\|_2}$ in Theorem \ref{thm:thm-linear-z} and throughout the paper later on. While many of our results hold for $s<\frac{1}{\|A\|_2}$, assuming $s\leq \frac{1}{2\|A\|_2}$ makes it easy to convert between $P_s$ norm and $\ell_2$ norm, as stated in Lemma~\ref{lem:lem-norm}:

\begin{lem}\label{lem:lem-norm}
    Suppose $s\leq \frac{1}{2\|A\|_2}$. It holds for any $z\in \mathbb{R}^{m+n}$ that:
    \begin{equation*}
        \sqrt{\frac{1}{2s}}\|z\|_2\leq \|z\|_{P_s} \leq \sqrt{\frac{2}{s}}\|z\|_2 \ .
    \end{equation*}
\end{lem}

\subsection{Sharpness of linear inequality system}\label{sec:homo-cone}

The existing linear convergence results of PDHG for LP~\cite{lu2022infimal,applegate2023faster} heavily depend on the Hoffman constant of the KKT system. However, it is well-known that the global Hoffman constant is generally very loose and cannot be easily characterized. Recently,~\cite{pena2023easily} presents a simple and tight bound on the Hoffman constant (i.e., inverse of sharpness constant) for homogeneous linear inequality system. We here present these characterizations and discuss their fundamental differences.


First, recall the definition of the sharpness constant of a linear inequality system:
\begin{mydef}
    Consider a linear inequality system $F x= g$, $\tilde F x \le \tilde g$ and its solution set $\mathcal X^*=\{x\in \mathbb R^n: F x= g, \tilde F x \le \tilde g\}$. We call $\alpha>0$ a sharpness constant to the system if it holds for any $x\in \mathbb R^n$ that
    $$
    \alpha \mathrm{dist}(x, \mathcal X^*) \le \left\Vert \begin{pmatrix}
                F x-g \\ (\tilde Fx-\tilde g)^+
            \end{pmatrix} \right\Vert_2\ .
    $$
\end{mydef}
The results of Hoffman~\cite{hoffman1952approximate,pena2021new} show the existence of a sharpness constant $\alpha$ for any linear inequality system. The value of sharpness constant $\alpha$ is the inverse of the Hoffman constant to the linear inequality system.

Indeed, it is highly difficult to provide simple characterization to sharpness constant $\alpha$. One characterization to $\alpha$ for system $Fx=g, \tilde Fx\leq \tilde g$ is described in \cite{pena2021new}:
\begin{equation}\label{eq:hoffman}
    \alpha = \min_{J\in\mathcal S(F,\tilde F)}\min_{\substack{v\in\mathbb R_+^J,\; z\in F(\mathbb R^{n})\\ \|(v,z)\|_2=1,\; \tilde F_J^Tv+Fz-u=0}}\|u\|_2 \ ,
\end{equation}
where $\mathcal S(F,\tilde F)=\{ J\subset [m]:\{(\tilde Fx+s,Fx):s\in \mathbb R^m,s_J\geq 0,x\in\mathbb R^n \} \mathrm{\; is\;a\;linear\;subspace}\}$.
Informally speaking, the inner minimization in \eqref{eq:hoffman} computes an extension of minimal positive singular value for a certain matrix, which is specified by an ``active set'' $J$ of constraints. The outer minimization takes the minimum of these extended minimal positive singular values over all possible ``active sets'' (intuitively, it goes through every extreme point, edge, face, etc., of $\mathcal{X}^*$).
While the inner minimization is expected when characterizing the behaviors of an algorithm, similar to the strong convexity in a minimization problem, there are usually exponentially many ``active sets'' for a polytope $\mathcal{X}^*$, thus $\alpha$ defined in \eqref{eq:hoffman} is known to be a loose bound and it is generally NP-hard to compute its exact value or even just a reasonable bound.


On the other hand, it turns out a simple characterization exists for the sharpness of a \emph{homogeneous} linear inequality system as shown recently in \cite{pena2023easily}. It does not require going through the exponential many ``active sets'' as in \eqref{eq:hoffman} and dramatically simplifies the characterization of the sharpness constant. 

\begin{prop}[Proposition 1-3, Theorem 1 in \cite{pena2023easily}]\label{prop:javier}
    Consider a linear inequality system $Fx=0$, $\tilde Fx\leq 0$ where $F\in\mathbb R^{m\times n}$ and $\tilde F\in\mathbb R^{\tilde m\times n}$, and its solution set $\mathcal X^*=\{x\in \mathbb R^n: F x= 0, \tilde F x \le 0\}$. Denote $\alpha$ the sharpness constant to the system, i.e., it holds for any $x\in \mathbb R^n$ that
    $
    \alpha \mathrm{dist}(x, \mathcal X^*) \le \left\Vert \begin{pmatrix}
                F x-0 \\ (\tilde Fx-0)^+
            \end{pmatrix} \right\Vert_2\ .
    $ Then:
    
    (a) There exists a unique partition $P\cup Q=\{1,...,\tilde m\}$ such that for $\tilde F=\begin{pmatrix}
        \tilde F_{P} \\ \tilde F_{Q}
    \end{pmatrix}$,
    \begin{align*}
        \begin{split}
            Fv=0, \; \tilde F_{P}v=0,\; \tilde F_{Q}v<0 \;for\;some\; v\in \mathbb R^{n} \ ,
        \end{split}
    \end{align*}
    and
    \begin{equation*}
        F^Tw_1+\tilde F_{P}^Tw_2 =  0\; for\; some \;w_1,w_2>0 \ .
    \end{equation*}

    (b) Moreover, define
    \begin{equation*}
        L:=\{ v: Fv=0,\; \tilde F_{P}v=0  \},\; K:=\{ v: \tilde F_{Q}v\leq 0 \} \ .    
    \end{equation*}
    It holds that
    \begin{equation}\label{eq:homo-sharp-new}
        \alpha(L,K)\cdot \min\{ \alpha_0(L),\alpha_0(K) \} \leq \alpha \leq \min\{ \alpha_0(L),\alpha_0(K) \} \ ,
    \end{equation}    
    where
    \begin{align}\label{eq:homo-sharp-1}
        \begin{split}
            \alpha_0(K)\geq\min_{w\geq 0,\|w\|_2=1}\| \tilde F_{Q}^Tw\|_2,\; \alpha_0(L)\geq\min_{\|v\|_2=1,Fv-w_1\leq 0,\tilde F_{P}v-w_2\leq 0} \| (w_1,w_2)\|_2 \ ,
        \end{split}
    \end{align}
    and
    \begin{equation}\label{eq:homo-sharp-2}
        \alpha(L,K) = \inf_{u\notin L\cap K} \frac{\max\{ \mathrm{dist}(u,L),\mathrm{dist}(u,K)\}}{\mathrm{dist}(u,L\cap K)} > 0 \ .
    \end{equation}
\end{prop}

    


\begin{rem}
    Similar to the interpretation of the inner minimum in \eqref{eq:hoffman}, $\alpha_0(K)$ and $\alpha_0(L)$ can be interpreted as an extension of the minimal non-zero singular value of the matrices. They are indeed highly related to and can be lower-bounded by the distance-to-infeasibility defined in Renegar's condition number~\cite{renegar1995incorporating}.
    $\alpha(L,K)$ essentially measures the angle between $K$ and $L$ and it is strictly positive by the construction of $L$ and $K$~\cite{pena2023easily}.
    
\end{rem}

Compared to the characterization of sharpness constant of a non-homogeneous linear system~\eqref{eq:hoffman}, which requires computing the minimum over potentially exponentially many values, the bound in \eqref{eq:homo-sharp-new} for homogeneous linear system essentially just looks at the corresponding values of two sub-systems and combines them using the angle measurement $\alpha(L,K)$. As a result, the sharpness constant of a homogeneous system is much simpler and tighter than the non-homogeneous counterpart, and can be computed in polynomial time~\cite{pena2023easily}.



\section{Stage I: finite time identification}\label{sec:stage-1}
In this section, we characterize the initial slow convergence stage and the phase transition between the two stages as described in Section \ref{sec:intro}. It turns out that PDHG can identify the ``active basis'' of the converging optimal solution in the first stage. We show that this stage always finishes (i.e., the phase transition happens) within a finite number of iterations, and the length of the first phase is driven by a quantity that characterizes how close the non-degeneracy part of the optimal solution is to degeneracy. 

Such phase transition behavior is closely related to the concept of partial smoothness that was proposed in previous literature to understand non-smooth optimization and primal-dual algorithms~\cite{wright1993identifiable,lewis2011identifying,oberlin2006active,hare2004identifying,hare2007identifying,lewis2013partial,daniilidis2014orthogonal,lewis2022partial,liang2018local}. In contrast to the previous literature, we do not assume the non-degeneracy condition of the problem. In the context of LP, the non-degeneracy condition in partial smoothness refers to that the iterates converge to an optimal solution which satisfies strict complimentary slackness. Such condition is also called ``irrepresentable condition'' in the optimization  literature~\cite{fadili2018sensitivity,fadili2019model}, because it is highly rare for first-order methods to converge to a strictly complementary optimal solution in practice.

{\blue Indeed, the major difficulty of our analysis is to handle the degeneracy of the problem. To the best of our knowledge, this work is the first complexity result on finite-time identification without the non-degeneracy assumption.}


To describe the identification property of the algorithm, we first introduce some new notations.


\begin{mydef}\label{def:partition}
For a linear programming~\eqref{eq:minmax} and an optimal primal-dual solution $z^*=(x^*,y^*)$, we define the index partition $(N,B_1,B_2)$ of the primal variable coordinates as
    \begin{align*}
     \ N&=\{i\in [n]: c_i-A_i^Ty^*>0\} \\ \ B_1&=\{i\in [n]: c_i-A_i^Ty^*=0, x_i^*>0\} \\ \  B_2&=\{i\in [n]: c_i-A_i^Ty^*=0, x_i^*=0\} \ .
\end{align*}
\end{mydef}


Since $(x^*,y^*)$ is an optimal solution pair to the LP, it follows from complementary slackness that $x^*_N=0$. However, $(x^*,y^*)$ may not satisfy strict complementary slackness, i.e., $B_2$ can be a non-empty set. The lack of strict complementary slackness is called degeneracy herein. 

In analogy to the notions in simplex method, we call the elements in $N$ the non-basic variables, the elements in $B_1$ the non-degenerate basic variables, and the elements in $B_2$ the degenerate basic variables. Furthermore, we call the elements in $N\cup B_1$ the non-degenerate variables and the elements in $B_2$ the degenerate variables. We also denote $B=B_1\cup B_2$ as the set of basic variables. 

Next, we introduce the non-degeneracy metric $\delta$, which essentially measures how close the non-degenerate part of the optimal solution is to degeneracy:

\begin{mydef}\label{def:delta}
For a linear programming~\eqref{eq:minmax} and an optimal primal-dual solution pair $z^*=(x^*,y^*)$, we define the non-degeneracy metric $\delta$ as:
\begin{equation*}
    \delta := 
    \min\left\{\min_{i\in N} \frac{c_i-A_i^Ty^*}{\|A\|_2},\;\min_{i\in B_1} x_i^*\right\} \ .
\end{equation*}
\end{mydef}

    {\blue The non-degeneracy metric $\delta$ looks at the non-degenerate variables (i.e. $N \cup B_1$) and measures how close they are to degeneracy. For non-degenerate non-basic variables in $N$, it computes the minimal scaled reduced cost $\min_{i\in N}\frac{c_i-A_i^Ty^*}{\|A\|_2}$; for non-degenerate basic variables in $B_1$, it computes the minimal primal variable slackness $\min_{i\in B_1}x_i^*$; and then $\delta$ takes the minimum of the two. For every optimal solution $z^*$ to a linear programming, the non-degeneracy metric $\delta$ is always strictly positive. The smaller the $\delta$ is, the closer the non-degenerate part of the solution is to degeneracy. In the rest of this section, we show that the value of $\delta$ drives the identification of the algorithm.}

Next, we introduce the sharpness to a homogeneous conic system  $\alpha_{L_1}$ that arises in the complexity of identification. We can view it as a local sharpness condition of a homogeneous linear inequality system around the converging optimal solution. Following the results stated in Section \ref{sec:homo-cone}, $\alpha_L$ is an extension of the minimal non-zero singular value of a certain matrix and does not depend on the overly-conservative global Hoffman constant.


More formally, we consider the following homogeneous linear inequality system
\begin{equation}\label{eq:active-cone}
    A_B^Tv\leq 0,Au=0, u_{N\cup B_2}\geq 0, \frac 1R(c^Tu-b^Tv)\leq 0 \ ,
\end{equation}
and denote $\mathcal K:=\{(u,v)\;|\;A_B^Tv\leq 0,Au=0, u_{N\cup B_2}\geq 0, \frac 1R(c^Tu-b^Tv)\leq 0\}$ the feasible set of \eqref{eq:active-cone}.
Denote $\alpha_{L_1}$ the sharpness constant to \eqref{eq:active-cone}, i.e., for any $(u,v)\in\mathbb R^{m+n}$,
\begin{equation}\label{eq:def_alphaL}
    \alpha_{L_1}\mathrm{dist}((u,v),\mathcal K)\leq \left\|\begin{pmatrix}
        [A_B^Tv]^+ \\ Au \\ [-u_{N\cup B_2}]^+ \\ \frac 1   R[c^Tu-b^Tv]^+ \end{pmatrix}\right\|_2   \ .
\end{equation}


Theorem \ref{thm:thm-identification} presents the main result for this section. It essentially states that after a certain number of iterations, PDHG iterates are reasonably close to the converging optimal solution and the algorithm can identify the elements in set $N$ and $B_1$.
Furthermore, we present the complexity theory for identification in terms of the non-degenerate metric $\delta$ and the local sharpness constant $\alpha_{L_1}$. {\blue Notice that $\delta$ does not rely on the degenerate coordinates, i.e., those in $B_2$, which shows that degeneracy does not slow down the identification, in contrast to previous literature on partial smoothness.} We also comment that the set $B_2$ may not be identifiable by PDHG iterates due to the degeneracy of the problem.
\begin{thm}[Identification Complexity]\label{thm:thm-identification}
    Consider PDHG for solving \eqref{eq:minmax} with step-size $s\leq \frac{1}{2\|A\|_2}$. Let $\{z^k=(x^k,y^k)\}_{k=0}^{\infty}$ be the iterates of the algorithm and let $z^*$ be the converging optimal solution, i.e., $z^k \rightarrow z^*$. Denote $(N, B_1, B_2)$ the partition of the primal variable indices using Definition \ref{def:partition}. Denote $\alpha_{L_1}$ the sharpness constant to the homogeneous linear inequality system \eqref{eq:active-cone}. Denote  $R=2\pran{\Vert z^0-z^* \Vert_{2}+\Vert z^* \Vert_{2}}+1$. Then, it holds for any
    \begin{equation}\label{eq:eq-ident-bound}
        k\geq K:=\max\left\{4,\frac{1}{s^2\alpha_{L_1}^2}\right\}\frac{256R^2}{\delta^2}+\frac{2}{s\Vert A \Vert_2} \ ,
    \end{equation}
    that 
    \begin{equation*}
        \Vert z^k-z^* \Vert_2 \leq \frac{\delta}{2} \ , \ \text{and}\ \ \ 
        x^{k}_N = 0,\; x^k_{B_1}>0,\; c_N-A_N^Ty^k>0 \ .
    \end{equation*}
\end{thm}

\begin{rem}
    Theorem \ref{thm:thm-identification} shows that after 
        $K=\mathcal O\pran{\max\left\{1,\frac{1}{s^2\alpha_{L_1}^2}\right\}\frac{R^2}{\delta^2} + \frac{1}{s\Vert A \Vert_2}}$
    iterations, the solution $z^k$ is $\delta/2$-close to the convergent optimal solution $z^*$. Furthermore, the set $N$ and $B_1$ can be identified after $K$ iterations.
    
\end{rem}

The rest of this section presents the proof of Theorem \ref{thm:thm-identification}. First, we consider the linear inequality system
    \begin{equation}\label{eq:active}
    A_B^Ty\leq c_B,\; Ax=b,\; x_{N\cup B_2}\geq 0,\; \frac 1R(c^Tx-b^Ty)\leq 0 \ ,
\end{equation}
and denote $\mathcal Z_L^*:=\{(x, y)\;|\;A_B^Ty\leq c_B,\; Ax=b,\; x_{N\cup B_2}\geq 0,\; \frac 1R(c^Tx-b^Ty)\leq 0\}$ the feasible set to \eqref{eq:active}. 

The next lemma builds up the connections between $\mathcal Z^*_L$ and $\mathcal{K}$. More specifically, it states that (a) $\mathcal Z^*_L$ is a shift of the cone $\mathcal{K}$, (b) the sharpness constant of the two linear inequality systems~\eqref{eq:active-cone} and~\eqref{eq:active} are the same, (c) one can lower bound $\mathrm{dist}(0,\mathcal F(z))$ using $\mathrm{dist}(z,\mathcal Z^*_L)$ and the sharpness constant of the systems.
\begin{lem}\label{lem:better-hoffman}
(a). It holds that 
    \begin{equation*}
        \mathcal Z^*_L = z^* +\mathcal K \ .
    \end{equation*}

(b). $\alpha_{L_1}$ is the sharpness constant to system  \eqref{eq:active}, that is, for any $z=(x,y)\in \mathbb R^{m+n}$,
    \begin{equation*}
        \alpha_{L_1}\mathrm{dist}(z,\mathcal Z^*_L) \leq \left\|\begin{pmatrix}
        [A_B^Ty-c_B]^+ \\ Ax-b \\ [-x_{N\cup B_2}]^+ \\ \frac 1   R[c^Tx-b^Ty]^+ \end{pmatrix}\right\|_2   \ .
    \end{equation*}

(c). For any $z\in B_R(0)\cap \{x:x\geq 0 \}$ we have
    \begin{equation*}
        \alpha_{L_1}\mathrm{dist}(z,\mathcal Z^*_L) \leq \mathrm{dist}(0,\mathcal F(z)) \ .
    \end{equation*}
\end{lem}

\begin{proof}
    (a). For any $z\in z^*+\mathcal K$, there exists $w=(u,v)\in \mathcal K$ such that $z=z^*+w$. By the definition of set $N, B, B_1, B_2$, we have
    \begin{align*}
        \begin{split}
            & \ A_B^T(y^*+v)-c_B = A_B^Tv \leq 0 \\
            & \ A(x^*+u) = b +Au = b \\
            & \ (x^*+u)_{N\cup B_2} = u_{N\cup B_2} \geq 0 \\
            & \ \frac 1R\pran{c^T(x^*+u)-b^T(y^*+v)} = \frac 1R(c^Tu-b^Tv) \leq 0 \ ,
        \end{split}
    \end{align*}
    which implies
    \begin{equation}\label{eq:eq-subset-1}
        z^*+\mathcal K \subset \mathcal Z^*_L \ .
    \end{equation}
    On the other hand, for any $z\in \mathcal Z^*_L$, we have $z=z^*+(z-z^*)$. Notice that
    \begin{align*}
        \begin{split}
            & \ A_B^T(y-y^*) = A_B^Ty-c_B \leq 0 \\
            & \ A(x-x^*) = Ax-b = 0 \\
            & \ (x-x^*)_{N\cup B_2} = x_{N\cup B_2} \geq 0 \\
            & \ \frac{1}{R}\pran{c^T(x-x^*)-b^T(y-y^*)} = \frac{1}{R}(c^Tx-b^Ty) \leq 0 \ ,
        \end{split}
    \end{align*}
    thus $z-z^*\in \mathcal K$, which implies
    \begin{equation}\label{eq:eq-subset-2}
         \mathcal Z^*_L \subset z^*+\mathcal K \ .
    \end{equation}
    Combining \eqref{eq:eq-subset-1} and \eqref{eq:eq-subset-2}, we arrive at
    \begin{equation*}
        \mathcal Z^*_L = z^*+\mathcal K \ .
    \end{equation*}
    
    (b). By definition of $\alpha_{L_1}$, we have
    \begin{align*}
    \begin{split}
        \alpha_{L_1}\mathrm{dist}(z,\mathcal Z^*_L) & \ = \alpha_{L_1}\mathrm{dist}(z-z^*,\mathcal Z^*_L-z^*)\\
        & \ = \alpha_{L_1}\mathrm{dist}(z-z^*,\mathcal K) \\
        & \ \leq \left\|\begin{pmatrix}
        [A_B^T(y-y^*)]^+, A(x-x^*), [-(x-x^*)_{N\cup B_2}]^+ , \frac 1R[c^T(x-x^*)-b^T(y-y^*)]^+ \end{pmatrix}\right\|_2 \\
        & \ = \left\|\begin{pmatrix}
        [A_B^Ty-c_B]^+, Ax-b, [-x_{N\cup B_2}]^+ , \frac 1R[c^Tx-b^Ty]^+ \end{pmatrix}\right\|_2   \ .
    \end{split}
    \end{align*}

(c). From the definition of sharpness and Lemma \ref{lem:better-hoffman}, it holds for any $z\in B_R(0)$ with $x\geq 0$ that
    \begin{align*}
    \begin{split}
        \alpha_{L_1}\mathrm{dist}(z,\mathcal Z^*_L) & \ \leq \left\|\begin{pmatrix}
        [A_B^Ty-c_B]^+, Ax-b, \frac 1R[c^Tx-b^Ty]^+ \end{pmatrix}\right\|_2\\
        & \ \leq \left\|\begin{pmatrix}
        [A^Ty-c]^+, Ax-b, \frac 1R[c^Tx-b^Ty]^+ \end{pmatrix}\right\|_2\\
        & \ \leq \rho_r(z)  \leq  \mathrm{dist}(0,\mathcal F(z)) \ ,
    \end{split}
    \end{align*}
    where the third and fourth inequalities follow from Proposition \ref{prop:prop-res-gap} and \cite[Proposition 1]{lu2021linear} respectively.
\end{proof}


The next lemma shows when $z$ is close enough to $z^*$, the non-basic dual slacks in set $N$ and the non-degenerate basic variables in set $B_1$ are bounded away from zero.
\begin{lem}\label{lem:lem-strict-delta}
    For any $z$ such that
        $\Vert z-z^* \Vert_2 \leq \frac{\delta}{2} \ ,$
    it holds for any $i\in N$, $j\in B_1$ that 
    \begin{equation*}
        c_i-A_i^Ty \geq \frac{\delta}{2}\Vert A \Vert_2,\; x_j\geq \frac{\delta}{2} \ .
    \end{equation*}
\end{lem}
\begin{proof}
    It follows from $\Vert z-z^* \Vert_2 \leq \frac{\delta}{2}$ that
    \begin{align*}
        \begin{split}
            \ \Vert x-x^* \Vert_2 \leq \frac{\delta}{2} \ , \Vert y-y^* \Vert_2 \leq \frac{\delta}{2} \ .
        \end{split}
    \end{align*}
    As a result, it holds for any $j\in B_1$ and $i\in N$ that
    \begin{align*}
        \begin{split}
            & \ |x_j-x_j^*|\leq \Vert x-x^* \Vert_2 \leq \frac{\delta}{2} \ , \\
            & \ |c_i-A_i^Ty-c_i-A_i^Ty^*| \leq \Vert (c-A^Ty)-(c-A^Ty^*) \Vert_2 \leq \Vert A \Vert_2 \Vert y-y^* \Vert_2 \leq \Vert A \Vert_2\frac{\delta}{2} \ ,
        \end{split}
    \end{align*}
    and thus by definition of $\delta$, we have
    \begin{align*}
        \begin{split}
            & \ x_j\geq  -\frac{\delta}{2}+x_j^* \geq \frac{\delta}{2} \ , \\
            & \ c_i-A_i^Ty \geq -\Vert A \Vert_2\frac{\delta}{2}+(c_i-A_i^Ty^*) \geq \Vert A \Vert_2(-\frac{\delta}{2}+\delta) = \Vert A \Vert_2\frac{\delta}{2} \ .
        \end{split}
    \end{align*}
\end{proof}
The next lemma connects $\mathcal Z^*_L$ and the optimal solution set $\mZ^*$.
\begin{lem}\label{fact:opt}
    $\mathcal Z^*_L \cap B_{\delta}(z^*) \subset \mathcal Z^*$.
\end{lem}
\begin{proof}
    Following the same proof of Lemma \ref{lem:lem-strict-delta}, we know that for any $z\in B_{\delta}(z^*)$, 
    \begin{equation}\label{eq:local-global-set-1}
        c_N-A_N^Ty\geq 0,\; x_{B_1}\geq 0 \ .
    \end{equation}
    Recall the definition of $\mathcal Z_L^*$ as the solution set to the system
    \begin{equation}\label{eq:local-global-set-2}
        A_B^Ty-c_B\leq 0,\; Ax=b,\; x_{N\cup B_2}\geq 0,\; \frac 1R(c^Tx-b^Ty)\leq 0 \ . 
    \end{equation}
    Note that combining \eqref{eq:local-global-set-1} and \eqref{eq:local-global-set-2} gives exactly the KKT system of original LP \eqref{eq:minmax} and thus
    \begin{equation*}
        \mathcal Z^*_L \cap B_{\delta}(z^*) \subset \mathcal Z^* \ .
    \end{equation*}
\end{proof}


The following lemma ensures that after a certain number of iterations, the iterates $\{z^k\}$ stay in the ball centered at $z^*$ with radius $\frac{\delta}{2}$. 

\begin{lem}\label{lem:global-local}
    Under the same condition of Theorem \ref{thm:thm-identification}, for any iteration $k$ such that
    \begin{equation*}
        k\geq K_1:=\max\left\{4,\frac{1}{s^2\alpha_{L_1}^2}\right\}\frac{256\mathrm{dist}^2(z^0,\mathcal Z^*)}{\delta^2} \ ,
    \end{equation*}
    we have
    \begin{equation*}
        \Vert z^k-z^* \Vert_2 \leq \frac{\delta}{2} \ .
    \end{equation*}
\end{lem}

\begin{proof}
    We first show that for any $k\ge K_1$, it holds that $\|z^{k+1}-z^k\|_{2}\leq \frac{\delta}{16}$ and $\mathrm{dist}(z^k,\mathcal Z^*_L)\leq \frac{\delta}{8}$, then we show $\| z^k-z^*\|_{2}\leq \frac{\delta}{2}$. The proof  heavily relies on the non-expansive nature of PDHG iterates.

    Notice that
    \begin{align*}
        \begin{split}
            & \ \|z^{k+1}-z^k\|_{2}  \leq \sqrt{2s}\|z^{k+1}-z^k\|_{P_s} 
            \leq \frac{\sqrt{2s}\mathrm{dist}_{P_s}(z^0,\mathcal Z^*)}{\sqrt k}\leq \frac{2\mathrm{dist}(z^0,\mathcal Z^*)}{\sqrt k}  \ , \\
        \end{split}
    \end{align*}
    where the first and third inequalities use Lemma \ref{lem:lem-norm} while the second one follows from part (a) in Theorem \ref{thm:thm-linear-z}. Furthermore, it holds that
    \begin{align*}
        \begin{split}
            & \ \mathrm{dist}(z^k,\mathcal Z^*_L)  \leq \frac{1}{\alpha_{L_1}}\mathrm{dist}(0,\mathcal F(z^k)) \leq \frac{1}{\alpha_{L_1}}\sqrt{\frac 2s}\mathrm{dist}_{P_s^{-1}}(0,\mathcal F(z^k))\leq \sqrt{\frac 2s}\frac{\mathrm{dist}_{P_s}(z^0,\mathcal Z^*)}{\alpha_{L_1}\sqrt k}\leq \frac{2\mathrm{dist}(z^0,\mathcal Z^*)}{s\alpha_{L_1}\sqrt k} \ ,
        \end{split}
    \end{align*}
    where the first inequality follows from part (c) in Lemma \ref{lem:better-hoffman}, the second and fourth inequalities utilize Lemma \ref{lem:lem-norm}, and the third inequality is from part (b) of Theorem \ref{thm:thm-linear-z}.     Therefore, we have $\|z^{k+1}-z^k\|_{2}\leq \frac{\delta}{16}$ and $\mathrm{dist}(z^k,\mathcal Z^*_L)\leq \frac{\delta}{8}$ whenever iteration number satisfies $k\ge K_1$. 
    
    In the rest of the proof, we consider two cases, separating by whether $\|P_{\mathcal Z^*_L}(z^k)-z^*\|_2\leq \frac{\delta}{2}$ or not. We will show that the first case implies $\| z^k-z^*\|_{2}\leq \frac{\delta}{2}$, and it is impossible for the second case to happen, which then finishes the proof of the lemma.
    

{\bf Case I: $\|P_{\mathcal Z^*_L}(z^k)-z^*\|_2\leq \frac{\delta}{2}$.} In this case, we have $P_{\mathcal Z^*_L}(z^k)\in \mathcal Z^*_L\cap B_{\delta/2}(z^*)\subset\mathcal Z^*$ from Lemma \ref{fact:opt} and therefore it follows from part (c) in Lemma \ref{lem:property} that
\begin{equation*}
    \| z^k-P_{\mathcal Z^*_L}(z^k)\|_{P_s} \geq \| z^*-P_{\mathcal Z^*_L}(z^k)\|_{P_s} \ .
\end{equation*}

Thus, it holds that
\begin{equation*}
    \| z^k-z^*\|_{P_s}\leq \| z^k-P_{\mathcal Z^*_L}(z^k)\|_{P_s}+\| P_{\mathcal Z^*_L}(z^k)-z^*\|_{P_s} \leq 2\| z^k-P_{\mathcal Z^*_L}(z^k)\|_{P_s}\leq \frac{2\sqrt 2}{\sqrt s}\| z^k-P_{\mathcal Z^*_L}(z^k)\|_{2} \ .
\end{equation*}

whereby
\begin{equation*}
    \| z^k-z^*\|_{2}\leq \sqrt{2s}\| z^k-z^*\|_{P_s}\leq 4\| z^k-P_{\mathcal Z^*_L}(z^k)\|_{2}= 4\mathrm{dist}(z^k,\mathcal Z^*_L) \leq \frac{\delta}{2} \ .
\end{equation*}

{\bf Case II: $\|P_{\mathcal Z^*_L}(z^k)-z^*\|_2>\frac{\delta}{2}$.} We aim to show this case is impossible under conditions $\|z^{k+1}-z^k\|_{2}\leq \frac{\delta}{16}$ and $\mathrm{dist}(z^k,\mathcal Z^*_L)\leq \frac{\delta}{8}$. 

First, since the iterates converge to an optimal solution (i.e. $z^k \rightarrow z^*$), we know that $\|z^k-z^*\|_2\geq \|P_{\mathcal Z^*_L}(z^k)-z^*\|_2>\frac{\delta}{2}$ and $\| z^k-z^*\|_2$  decreases to zero. Then there exists the smallest $N\geq k$ such that 
\begin{equation*}
    \|z^N-z^*\|_2>\frac{\delta}{2},\;\|z^{N+1}-z^*\|_2\leq \frac{\delta}{2} \ ,
\end{equation*}
and hence
\begin{equation}\label{eq:mix1}
    \|P_{\mathcal Z^*_L}(z^N)-z^*\|_2>\frac{\delta}{2},\;\|P_{\mathcal Z^*_L}(z^{N+1})-z^*\|_2\leq \frac{\delta}{2} \ ,
\end{equation}
where the first one is due to the conclusion in Case I and the second one is the consequence of non-expansiveness of projection. It then follows from  Lemma \ref{fact:opt} that $P_{\mathcal Z^*_L}(z^{N+1})\in\mathcal Z^*_L\cap B_{\delta/2}(z^*)\subset\mathcal Z^*$.

If $P_{\mathcal Z^*_L}(z^{N+1})=z^*$, then from \eqref{eq:mix1} we have 
\begin{equation}\label{eq:coincide}
    \|P_{\mathcal Z^*_L}(z^N)-P_{\mathcal Z^*_L}(z^{N+1})\|_2=\|P_{\mathcal Z^*_L}(z^N)-z^*\|_2>\frac{\delta}{2}
\end{equation}
Otherwise, denote $\tilde z=z^*+\frac{\delta}{2\|P_{\mathcal Z^*_L}(z^{N+1})-z^*\|_2}(P_{\mathcal Z^*_L}(z^{N+1})-z^*)\in \mathcal Z^*_L$ and then $\|\tilde z-z^*\|_2=\frac{\delta}{2}$, which implies $\tilde z\in\mathcal Z^*_L\cap B_{\delta/2}(z^*)\subset\mathcal Z^*$. By triangle inequality, we have
\begin{align}\label{eq:mix2}
    \begin{split}
        \| P_{\mathcal Z^*_L}(z^{N})-\tilde z \|_2\geq\|z^N-\tilde z\|_2-\|P_{\mathcal Z^*_L}(z^{N})-z^N\|_2 \geq \frac{\delta}{4}-\mathrm{dist}(z^N,\mathcal Z^*_L)\geq \frac{\delta}{8}  \ ,
    \end{split}
\end{align}
where the second inequality follows from 
\begin{equation*}
    \|z^N-\tilde z\|_2\geq \sqrt{\frac{s}{2}}\|z^N-\tilde z\|_{P_s}\geq \sqrt{\frac{s}{2}}\|z^*-\tilde z\|_{P_s}\geq \sqrt{\frac{s}{2}}\sqrt{\frac{1}{2s}}\|z^*-\tilde z\|_{2}=\frac 12\times\frac{\delta}{2}=\frac{\delta}{4} \ ,
\end{equation*}
and $\|P_{\mathcal Z^*_L}(z^{N})-z^N\|_2=\mathrm{dist}(z^N,\mathcal Z^*_L) \leq \frac{\delta}{8}$ as $N\geq k \ge K_1$.

Consider the two-dimensional plane that $z^*$, $\tilde z$ and $P_{\mathcal Z^*_L}(z^{N})$ lay (see Figure \ref{fig:triangle}). We know by construction of $\tilde z$ that $P_{\mathcal Z^*_L}(z^{N+1})$ lies on the segment between  $z^*$ and $\tilde z$. Let $\|P_{\mathcal Z^*_L}(z^N)-z^*\|_2=t\frac{\delta}{2}$ and $\| P_{\mathcal Z^*_L}(z^{N})-\tilde z \|_2=w\frac{\delta}{8}$, then we have $t\geq 1$ and $w\geq 1$ by \eqref{eq:mix1} and \eqref{eq:mix2}. Define $\theta$ as the angle between vector $P_{\mathcal Z^*_L}(z^{N})-z^*$ and $\tilde z-z^*$, and $\phi$ as the angle between vector $P_{\mathcal Z^*_L}(z^{N})-\tilde z$ and $z^*-\tilde z$. There are three cases:

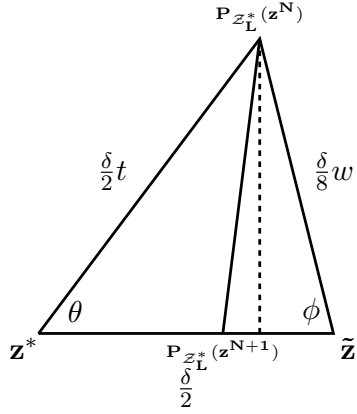
\begin{figure}
\centering
  \begin{tikzpicture}[scale=0.7,every node/.style={scale=1.5}]
        \begin{axis}[hide axis,xlabel={},ylabel={}, ytick={}, xtick={}, yticklabels={}, xticklabels={},legend pos=outer north east, legend cell align=left, xmin=-0.45, xmax=2.45, ymin=-0.5, ymax=2.4, y=2.8cm,  x=2.8cm,]
        \addplot[forget plot,color=black,line width=1.5pt] coordinates {(0,0) (2,0) (1.5,2) (0,0)};
        \addplot[forget plot,color=black,line width=1.5pt] coordinates {(1.25,0) (1.5,2)};
        \addplot[forget plot,color=black,line width=1.5pt, dashed] coordinates {(1.5,0) (1.5,2)};
        \node at (axis cs: 0.25,0.15) {$\theta$};
        \node at (axis cs: 1.85,0.15) {$\phi$};
        \node at (axis cs: 0.5,1.1) {$\frac{\delta}{2} t$};
        \node at (axis cs: 2,1.1) {$\frac{\delta}{8} w$};
        \node at (axis cs: 1,-0.37) {$\frac{\delta}{2}$};
        \node at (axis cs: 1.5,2.15) {\tiny $\mathbf{P_{\mathcal Z^*_L}(z^N)}$};
        \node at (axis cs: 1.25,-0.15) {\tiny $\mathbf{\small P_{\mathcal Z^*_L}(z^{N+1})}$};
        \node at (axis cs: -0.1,-0.1) {$\mathbf{z^*}$};
        \node at (axis cs: 2.1,-0.1) {$\mathbf{\tilde z}$};
        \end{axis}
        \end{tikzpicture}
\caption{Illustration of the proof of Case II in Lemma \ref{lem:global-local}}
\label{fig:triangle}
\end{figure}


(a) If $\mathrm{cos}(\theta)\leq 0$, then we have $\|P_{\mathcal Z^*_L}(z^{N})-P_{\mathcal Z^*_L}(z^{N+1})\|_2\geq \|P_{\mathcal Z^*_L}(z^{N})-z^*\| \geq \frac{\delta}{2}t\geq \frac{\delta}{2}$.

(b) If $\mathrm{cos}(\phi)\leq 0$, then we have $\|P_{\mathcal Z^*_L}(z^{N})-P_{\mathcal Z^*_L}(z^{N+1})\|_2\geq \|P_{\mathcal Z^*_L}(z^{N})-\tilde z\|\geq \frac{\delta}{8}w \geq \frac{\delta}{8}$.

(c) It $\mathrm{cos}(\theta)> 0$ and $\mathrm{cos}(\phi)> 0$, then it follows from law of cosine that
\begin{align*}
    \begin{split}
        & \ \frac{1}{16} w^2 = 1+t^2-2t\mathrm{cos}(\theta) \leq 1+t^2 \ ,
    \end{split}
\end{align*}
thus 
\begin{equation}\label{eq:eq-lb-t}
    t^2\geq \frac{1}{16}w^2-1\ .
\end{equation}
Again by law of cosine and note that $t\geq 1$ and \eqref{eq:eq-lb-t}, we have
\begin{align*}
    \begin{split}
        & \ 1+\frac{1}{16}w^2 -\frac{1}{2}w\mathrm{cos}(\phi)=t^2 \geq \max\left\{1,\frac{1}{16}w^2-1\right\} \ , \\ 
    \end{split}
\end{align*}
which implies
\begin{align*}
    \begin{split}
        & \ \mathrm{cos}(\phi)\leq \min\left\{\frac{w}{8}, \frac{4}{w} \right\} \leq \frac{\sqrt 2}{2} \ , \\ 
    \end{split}
\end{align*}
whereby
\begin{equation*}
    \mathrm{sin}(\phi)\geq \frac{\sqrt 2}{2} \ .
\end{equation*}

By two-dimensional geometry, we have
\begin{equation*}
    \|P_{\mathcal Z^*_L}(z^{N})-P_{\mathcal Z^*_L}(z^{N+1})\|_2 \geq \frac{\delta}{8}w \mathrm{sin}(\phi) \geq  \frac{\sqrt {2}}{16} \delta \ . 
\end{equation*}
Combining (a)-(c) and \eqref{eq:coincide}, we have
\begin{equation*}
    \|z^{N}-z^{N+1}\|_2 \geq \|P_{\mathcal Z^*_L}(z^{N})-P_{\mathcal Z^*_L}(z^{N+1})\|_2 > 
    \min\left\{ \frac{\delta}{2}, \frac{\delta}{8}, \frac{\sqrt 2\delta}{16}\right\}>\frac{\delta}{16} \ ,
\end{equation*}
However, since $N\geq k\geq K_1$, we know $\|z^{N+1}-z^N\|_{2}\leq \frac{\delta}{16}$. This leads to the contradiction.

To sum up, for any iteration $k$, when $\|z^{k+1}-z^k\|_{2}\leq \frac{\delta}{16}$ and $\mathrm{dist}(z^k,\mathcal Z^*_L)\leq \frac{\delta}{8}$, it must have $\|P_{\mathcal Z^*_L}(z^k)-z^*\|_2\leq \frac{\delta}{2}$ (i.e., Case II cannot happen) and thus $\| z^k-z^*\|_{2}\leq \frac{\delta}{2}$ following the argument in Case I, which completes the proof of Lemma \ref{lem:global-local}.
\end{proof}

By combining Lemma \ref{lem:lem-strict-delta} and \ref{lem:global-local}, we are ready to prove Theorem \ref{thm:thm-identification}.
\begin{proof}[Proof of Theorem \ref{thm:thm-identification}]
    Consider $k\ge K\ge K_1$. It follows from Lemma \ref{lem:global-local} that $\Vert z^k-z^* \Vert_2 \leq \frac{\delta}{2} $. Thus, we have from Lemma \ref{lem:lem-strict-delta} that $x_{B_1}^k > 0, c_N-A_N^T y^k >0$. We just need to show that $x^k_N=0$.
    Notice it holds from Lemma \ref{lem:lem-strict-delta} that for any $i\in N$ and $k\geq K_1$ 
    \begin{equation*}
        c_i-A_i^Ty^k \geq \Vert A \Vert_2\frac{\delta}{2} \ .
    \end{equation*}
    Consider the update formula of PDHG
    \begin{equation*}
        x^{k+1}_N=\mathrm{proj}_{\mathbb R_+^{|N|}}(x^{k}_N-s(c_N-A_N^Ty^k)) \ .
    \end{equation*}
    Thus, it follows from Lemma \ref{lem:lem-norm} that
    \begin{align*}
        \begin{split}
            \max_{i\in N}x^{K_1}_i \leq \Vert z_N^{K_1} \Vert_2 \leq \sqrt{2s} \Vert z_N^{K_1} \Vert_{P_s} = \sqrt{2s} \Vert z_N^{K_1}-z_N^* \Vert_{P_s} \leq \sqrt{2s} \Vert z^{K_1}-z^* \Vert_{P_s} \leq \delta  \ .
        \end{split}
    \end{align*}
    
    Then for any $q>\frac{2{\delta}}{s\delta\Vert A \Vert_2}=\frac{2}{s\Vert A \Vert_2}$ and $i\in N$, we have $\sum_{k=0}^{q-1} s(c-A^T y^{K_1+k})_i \ge \frac{q\delta s\|A\|_2}{2}$, thus it holds that
    \begin{equation*}
        x^{{K_1}+q}_N=0 \ ,
    \end{equation*}
    that is, after at most $K:=K_1+\frac{2}{s\Vert A \Vert_2}$ steps, the coordinates in set $N$ are fixed to zero. This finishes the proof of Theorem \ref{thm:thm-identification}.
\end{proof}

\section{Stage II: local convergence}\label{sec:stage-2}
In the previous section, we show that after a certain number of iterations, the iterates of PDHG stay close to an optimal solution, and the iterates identify the non-degenerate coordinates set $N$ and $B_1$. In this section, we characterize the local behaviors of PDHG for LP afterward. More specifically, we show that PDHG iterates converge to an optimal solution linearly after identification, and the local linear convergence rate is characterized by a local sharpness constant $\alpha_{L_2}$ for a homogeneous linear inequality system. The local convergence can be much faster than the global linear convergence rate obtained in~\cite{applegate2023faster}, which is characterized by the Hoffman constant of the KKT system of the LP. This provides a theoretical explanation of the eventual fast rate of PDHG for LP.

Throughout this section, we consider the iterates $z^k$ after identification, i.e., for $k\ge K$. In this stage, it holds that $\Vert z^k-z^* \Vert_2\leq \frac{\delta}{2}$ and $x^{k}_N = 0,\; x^k_{B_1}>0,\; c_N-A_N^Ty^k >0$ from Theorem \ref{thm:thm-identification}. We call this stage the local setting. To develop the eventual faster linear convergence rate in the local setting, we first reformulate the KKT system of LP to remove the non-basic variables. Then we show the faster local linear convergence rate using the sharpness of a homogeneous linear inequality system.






Recall that the KKT system for original linear programming \eqref{eq:minmax} is
\begin{equation}\label{eq:eq-original}
    \begin{cases}
        Ax=b,\; x\geq 0 \\
        A^Ty \leq c \\
        c^Tx-b^Ty\leq 0
    \end{cases} \ .
\end{equation}
Every solution that satisfies \eqref{eq:eq-original} is an optimal primal-dual solution pair to the original LP \eqref{eq:minmax}.
In the local setting, note that $x^{k}_N = 0, \; c_N-A_N^Ty^k >0$, \eqref{eq:eq-original} can be reduced to the following system by removing non-basic coordinates $N$:
\begin{equation}\label{eq:eq-pd-reduce}
    \begin{cases}
        A_Bx_B=b,\; x_B\geq 0 \\
        A_B^Ty \leq c_B \\
        c_B^Tx_B-b^Ty\leq 0
    \end{cases} \ .
\end{equation}
In particular, for any local iterate $z^k$ for $k\ge K$, if \eqref{eq:eq-pd-reduce} is satisfied, then it is an optimal solution to \eqref{eq:eq-original}.


Denote $\widetilde{\mathcal X_B^*}\times \widetilde{\mathcal Y^*}$ the solution set to \eqref{eq:eq-pd-reduce}. Recall that $z^*$ is the optimal solution that the algorithm converges to. By the definition of set $B$ and $N$, we have $z^*=(x^*,y^*)=(x^*_N, x^*_B, y^*)$, where $x^*_N=0$. Since $z^*$ satisfy \eqref{eq:eq-original}, it must hold that $(x_B^*,y^*)\in \widetilde{\mathcal X_B^*}\times \widetilde{\mathcal Y^*}$. Moreover, from the definition of set $N$ and $B_1, B_2$, it holds that
\begin{equation}\label{eq:eq-pd-active}
    \begin{cases}
        A_Bx_B^*=b,\; x_{B_1}^*> 0,\; x_{B_2}^*=0\\
        A_B^Ty^*=c_B \\
        c_B^Tx_B^*-b^Ty^*=0
    \end{cases} \ .
\end{equation} 

Next, we discuss the characterization of local sharpness properties for~\eqref{eq:eq-pd-reduce} and its critical role in the local convergence rate. The idea is to connect~\eqref{eq:eq-pd-reduce} in the local setting with a homogeneous conic system. The results explain why the local rate is much faster than the global rate that is dependent on the global Hoffman constant.


In this section, we aim to derive the local convergence complexity of PDHG on LP. 
To develop the eventual faster linear convergence rate, we first analyze the sharpness of the reduced primal-dual system of LP and build the connection with a homogeneous linear inequality system. Based on the understanding of local sharpness, a refined local convergence rate is then derived. 

We consider the following homogeneous linear inequality system
    \begin{equation}\label{eq:local-cone}
        A_Bu_B=0,\;u_{B_2}\geq 0,\; A_B^Tv\leq 0,\; \tfrac{1}{R}(c_B^Tu_B-b^Tv)\leq 0 \ .
    \end{equation}
    and denote $\mathcal W^*:=\left\{(u_B,v)\;|\;A_Bu_B=0,A_B^Tv\leq 0,\frac 1R (c_B^Tu_B-b^Tv)\leq 0, u_{B_2}\geq 0  \right\}$ the feasible set of \eqref{eq:local-cone}. {\blue Denote $R_2=\delta+\|z^*\|$.} Denote $\alpha_{L_2}$ the sharpness constant to \eqref{eq:local-cone}, i.e., for any $(u_B,v)\in\mathbb R^{|B|+m}$,
    \begin{equation*}
        \alpha_{L_2}\mathrm{dist}((u_B,v),\mathcal W^*)\leq\left\| \begin{pmatrix}
            A_Bu_B \\ [A_B^Tv]^+ \\ [-u_{B_2}]^+ \\ \textcolor{black}{\frac{1}{ R_2}}[c_B^Tu_B-b^Tv]^+
        \end{pmatrix}  \right\|_2 \ .
    \end{equation*}

{There is a slight difference of the sharpness constants $\alpha_{L_1}$ and $\alpha_{L_2}$, because they correspond to slightly different linear inequality systems, \eqref{eq:active-cone} and \eqref{eq:local-cone}, respectively. The difference is that in the Stage II (i.e., $\alpha_{L_2}$), the non-basic set $N$ is identified, i.e., $x_N^k=0$. Thus in Stage II, the iterates of PDHG are driven by sharpness $\alpha_{L_2}$ of system \eqref{eq:local-cone} which can ignore the constraints in non-basic set $N$, while sharpness $\alpha_{L_1}$ of \eqref{eq:active-cone} does depend on the whole constraint matrix. Nevertheless, $\alpha_{L_1}$ and $\alpha_{L_2}$ are sharpness of homogeneous linear inequality systems \eqref{eq:active-cone} and \eqref{eq:local-cone} respectively and thus exhibit simpler and tighter characterizations, as stated in Section \ref{sec:homo-cone}. {\blue We present a simple example in Appendix \ref{app:example} that demonstrates $\alpha_{L_2}$ can be arbitrarily larger than the global sharpness constant $\alpha$.}
}

Theorem \ref{thm:thm-local} presents the local linear convergence of PDHG after identification.
\begin{thm}[Local Linear Convergence]\label{thm:thm-local}
Consider PDHG for solving \eqref{eq:minmax} with step-size $s\leq \frac{1}{2\|A\|_2}$. Let $\{z^k=(x^k,y^k)\}_{k=0}^{\infty}$ be the iterates of the algorithm and let $z^*$ be the converging optimal solution, i.e., $z^k \rightarrow z^*$. Then it holds for any $k>K$ that
\begin{align*}
    \|z^k-z^*\|_2 \leq 4\delta\exp\pran{ -\frac{k-K}{2\lceil 4e/(s^2\alpha_{L_2}^2)\rceil} } \ ,
\end{align*}
where $K$ is the maximal length of the first stage defined in \eqref{eq:eq-ident-bound}.
\end{thm}

Theorem \ref{thm:thm-local} shows that after identification, PDHG can obtain an $\epsilon$-close optimal solution within $$\mathcal O\pran{\frac{1}{s^2\alpha_{L_2}^2}\log\pran{\frac{\delta}{\epsilon}}}$$ iterations. Combining Theorem \ref{thm:thm-identification} and Theorem \ref{thm:thm-local}, the complexity of PDHG finding an $\epsilon$-close optimal solution is
\begin{equation*}
        \mathcal O\pran{ \max\left\{1,\frac{1}{s^2\alpha_{L_1}^2}\right\}\frac{R^2}{\delta^2}  + \frac{1}{s\|A\|_2} + \frac{1}{s^2\alpha_{L_2}^2}\log\pran{\frac{\delta}{\epsilon}} } \ .
    \end{equation*}
    In particular, when setting step-size $s=\frac{C}{\Vert A\Vert_2}$ with $C\leq \frac 12$, the total complexity becomes
    { \begin{equation*}
        \mathcal O\pran{\max\left\{1,\frac{\|A\|_2^2}{\alpha_{L_1}^2}\right\}\frac{R^2}{\delta^2} +\frac{\Vert A\Vert_2^2}{\alpha_{L_2}^2}\log\pran{\frac{\delta}{\epsilon}} } \ .
    \end{equation*} }


The rest of this section presents a proof of Theorem \ref{thm:thm-local}. First, we show that $\alpha_{L_2}$ is a local sharpness constant of the non-homogeneous linear inequality system \eqref{eq:eq-pd-active} for $(x_B,y)$ that is close enough to $(x_B^*,y^*)$.

\begin{lem}\label{lem:prop-primal-hoffman}
It holds for any $(x_B,y)$ such that $\Vert (x_B,y)-(x_B^*,y^*) \Vert_2 \leq \frac{\delta}{2}$ and $x_{B}\geq 0$ that
    \begin{equation*}
        \alpha_{L_2}\mathrm{dist}((x_B,y),\widetilde{\mathcal X_B^*}\times \widetilde{\mathcal Y^*}) \leq \left\Vert \begin{pmatrix}
                A_Bx_B-b \\ [A_B^Ty-c_B]^+ \\ [\textcolor{black}{\frac{1}{R_2}}\pran{c_B^Tx_B-b^Ty}]^+
            \end{pmatrix} \right\Vert_2 \ .
    \end{equation*}
\end{lem}
\begin{proof}
Denote $\widetilde{\mathcal{Z}}=\{(x_B,y): \Vert (x_B,y)-(x_B^*,y^*) \Vert_2 \leq \frac{\delta}{2}, x_{B}\geq 0\}$. Notice by the definition of $\delta$, we have $x_{i}^*\geq\delta$ for any $i\in B_1$. Thus, it holds for any $(x,y)\in \widetilde{\mathcal{Z}}$ and $i\in B_1$ that $x_{i}\ge x_{i}^* - \Vert x_B-x_B^* \Vert_2\ge \frac{\delta}{2}>0$. In other words,  the constraint $\Vert (x_B,y)-(x_B^*,y^*) \Vert_2 \leq \frac{\delta}{2}$ implies $x_{B_1}\ge 0$. Therefore, we have
    \begin{equation}\label{eq:eq-degeneracy}
        \widetilde{\mathcal{Z}}=\{(x_B,y): \Vert (x_B,y)-(x_B^*,y^*) \Vert_2 \leq \tfrac{\delta}{2}, x_{B_2}\geq 0\} \ .
    \end{equation}
    
    Now denote $u_B=(u_{B_1},u_{B_2}):=x_B-x_B^*$ as the shift of $x_B$ from $x_B^*$, namely, $u_{B_1}=x_{B_1}-x_{B_1}^*$ and $u_{B_2}=x_{B_2}-x_{B_2}^*=x_{B_2}$ (by definition of $B_2$ we have $x_{B_2}^*=0$). Denote $v:=y-y^*$ as the shift of $y$ from $y^*$.  It then holds that $(u_B,v)\in \mathcal{W}^*:= \{(u_B,v)\;|\;\|(u_B,v)\|_2\leq \frac{\delta}{2}, u_{B_2}\geq 0\}$. 

    
    Next, notice $0\in \mathcal W^*$, thus $P_{\mathcal W^*}(0)=0$. It then follows from the nonexpansiveness of the projection operator that
    \begin{equation*}
        \| P_{\mathcal W^*}(u_B,v) \|_2 = \| P_{\mathcal W^*}(u_B,v)-P_{\mathcal W^*}(0) \|_2 \leq \| (u_B,v)-(0,0)\|_2=\|(u_B,v)\|_2\leq \delta/2 \ .
    \end{equation*}
    Thus, we have $P_{\mathcal W^*}(u_B,v)\in B_{\delta/2}(0)$ for any $(u_B,v)\in \mathcal W$. As a result, it holds for all $(u_B,v)\in \mathcal W$ that
    \begin{equation}\label{eq:w-proj}
        \mathrm{dist}((u_B,v),\mathcal W^*) =\mathrm{dist}((u_B,v),\mathcal W^*\cap B_{\delta/2}(0)) \ .
    \end{equation}
    Since $(x^*_B,y^*)$ satisfies \eqref{eq:eq-pd-active}, we have
    \begin{align}\label{eq:eq-shift}
        \begin{split}
        & \quad \  \Vert A_Bx_B-b \Vert_2^2 + \Vert [A_B^Ty-c_B]^+\Vert_2^2 + \Vert [\textcolor{black}{\tfrac{1}{R_2}}\pran{c_B^Tx_B-b^Ty}]^+\Vert_2^2 \\
        & = \Vert A_{B}(x_{B}-x_{B}^*) \Vert_2^2+ \Vert [A_B^T(y-y^*)]^+\Vert_2^2 + \Vert \textcolor{black}{\tfrac{1}{R_2}}[c_B^T(x_B-x_B^*)-b^T(y-y^*)]^+\Vert_2^2\\
        & = \Vert A_{B}u_{B} \Vert_2^2+\Vert [A_B^Tv]^+\Vert_2^2+\Vert \textcolor{black}{\tfrac{1}{R_2}}[c_B^Tu_B-b^Tv]^+\Vert_2^2 \ .
        \end{split}
    \end{align} 
    Furthermore, for any $(u_B,v)\in \mathcal W^*\cap B_{\frac{\delta}{2}}(0)$, we have $(x_B^*+u_B,y^*+v)\in \widetilde{\mathcal X_B^*}\times\widetilde{\mathcal Y^*}$, because
    \begin{align*}
        \begin{split}
            & \ A_B (x_B^*+u_B)-b=A_Bx_B^*-b+A_Bu_B=0\\
            & \ A_B^T(y^*+v)-c_B = A_B^Ty^*-c_B+A_B^Tv \leq 0\\
            & \ c_B^T(x_B^*+u_B)-b^T(y^*+v)=c_B^Tx_B^*-b^Ty^*+c_B^Tu_B-b^Tv\leq 0 \\
            & \ x_{B_1}^*+u_{B_1} \geq \tfrac{\delta}{2}\mathbf{1}_{|B_1|}> 0,\; x_{B_2}^*+u_{B_2}=u_{B_2} \geq 0 \ ,
        \end{split}
    \end{align*}
    thus it holds that
    \begin{equation}\label{eq:subset}
        (x_B^*,y^*)+ \mathcal W^*\cap B_{\frac{\delta}{2}}(0) \subset  \widetilde{\mathcal X_B^*}\times\widetilde{\mathcal Y^*} \ .
    \end{equation}
    
    Finally, by the definition of $\alpha_{L_1}$, we have
    {\small 
    \begin{align}\label{eq:alpha}
        \begin{split}
            \Vert A_{B}u_{B} \Vert_2^2+ \Vert [A_B^Tv]^+\Vert_2^2+ \Vert [\textcolor{black}{\tfrac{1}{R_2}}(c_B^Tu_B-b^Tv)]^+ \Vert_2^2 & \ \geq \alpha_{L_2}^2\mathrm{dist}^2((u_B,v),\mathcal W^*) \\
            & \ = \alpha_{L_2}^2\mathrm{dist}^2((u_B,v),\mathcal W^*\cap B_{\delta/2}(0)) \\
            & \ = \alpha_{L_2}^2\mathrm{dist}^2((x_B-x_B^*,y-y^*),(x_B^*,y^*)+\mathcal W^*\cap B_{{\delta}/{2}}(0)-(x_B^*,y^*))\\
            & \ =\alpha_{L_2}^2\mathrm{dist}^2((x_B,y),(x_B^*,y^*)+\mathcal W^*\cap B_{{\delta}/{2}}(0))\\
            & \ \geq \alpha_{L_2}^2\mathrm{dist}^2((x_B,y),\widetilde{\mathcal X_B^*}\times\widetilde{\mathcal Y^*}) \ ,
        \end{split}
    \end{align}
    }
    where the first equality uses \eqref{eq:w-proj} and the last inequality utilizes \eqref{eq:subset}.  We finish the proof by combining \eqref{eq:eq-shift} and \eqref{eq:alpha}.
\end{proof}

Now we are ready to prove Theorem \ref{thm:thm-local}.
\begin{proof}[Proof of Theorem \ref{thm:thm-local}]

Notice that $k>K$. It follows from Theorem \ref{thm:thm-identification} that $\Vert z^k-z^*\Vert_2\leq \frac{\delta}{2}$ and $x^k_N=0$. Denote the local normalized duality gap as
\begin{equation*}
    \rho^B_r(x_B,y)=\max_{(\hat x_B, \hat y)  \in W_r(x_B,y)}\frac{L^B(x_B,\hat y)-L^B(\hat x_B,y)}{r} \ ,
\end{equation*}
where $L^B(x_B,y)=c_B^Tx_B-y^TA_Bx_B+b^Ty$.
Then we have {\blue for any $r\in(0,\tilde R]$ that}
    \begin{align*}
        \begin{split}
            \alpha_{L_2} \mathrm{dist}((x^k,y^k),\mathcal X^*\times \mathcal Y^*) & \ \leq \alpha_{L_2}\mathrm{dist}((0_N,x_B^k,y^k),\{0_N\}\times\widetilde{\mathcal X_B^*}\times \widetilde{\mathcal Y^*})\\ 
            & \ = \alpha_{L_2}\mathrm{dist}((x_B^k,y^k),\widetilde{\mathcal X_B^*}\times \widetilde{\mathcal Y^*}) \\ 
            & \ \leq \left\Vert \begin{pmatrix}
                A_Bx_B^k-b \\ [A_B^Ty^k-c_B]^+ \\ [\textcolor{black}{\frac{1}{R_2}}\pran{c_B^Tx_B^k-b^Ty^k}]^+
            \end{pmatrix} \right\Vert_2\\
            & \ \leq  \rho^B_r(x_B^k,y^k)=\rho_r(x^k,y^k) \ ,
        \end{split}
    \end{align*}
    where the second inequality is due to definition of $\alpha_{L_2}$ and the last inequality utilizes Proposition \ref{prop:prop-res-gap} and \textcolor{black}{$\|(x_B^k,y^k)\|_2\leq \|(x^k,y^k)\|_2\leq \|(x^k,y^k)-(x^*,y^*)\|_2+\|(x^*,y^*)\|_2\leq \delta+\|z^*\|_2=R_2$ . }
    Applying part (b) of Theorem \ref{thm:thm-linear-z}, it holds for any $k>K$ that
\begin{align*}
    \|z^k-z^*\|_2& \ \leq \sqrt{2s}\|z^k-z^*\|_{P_s}\\
    & \ \leq 4\sqrt{2s}\exp\pran{ -\frac{k-K}{2\lceil 4e/(s^2\alpha_{L_2}^2)\rceil} }\mathrm{dist}_{P_s}((x^K,y^K),\mathcal X^*\times \mathcal Y^*)\\
    & \ \leq 8\exp\pran{ -\frac{k-K}{2\lceil 4e/(s^2\alpha_{L_2}^2)\rceil} }\frac{\delta}{2} \ ,
\end{align*}
which finishes the proof.
\end{proof}


\section{Numerical experiments}\label{sec:numerical}
In this section, we present numerical experiments that verify our theoretical results in the previous sections. 

{\bf Dataset.} In the experiments, we utilize the root-node LP relaxation of instances from the \href{https://miplib.zib.de/tag_collection.html}{\texttt{MIPLIB 2017}}. We convert the instances to the following form:
\begin{align}\label{eq:practical-lp-form}
    \begin{split}
        & \min_x\; c^Tx\\
        & \ \mathrm{s.t.}\; A_{E}x=b_E \\
        & \ \quad\;\; A_{I}x\leq b_I\\
        & \ \quad\;\; x\geq 0 \ ,
    \end{split}
\end{align}
and consider its primal-dual formulation
\begin{align}\label{eq:practical-lp}
    \begin{split}
        \min_{x\geq 0}\max_{y_I\leq 0}\; c^Tx-y^TAx+b^Ty \ , 
    \end{split}
\end{align}
where $A_E\in\mathbb R^{m_E\times n}$, $A_I\in \mathbb R^{m_I\times n}$, $A=\begin{pmatrix}
    A_{E} \\ A_{I}
\end{pmatrix}\in\mathbb R^{(m_E+m_I)\times n}$ and $b=\begin{pmatrix}
    b_E \\ b_I
\end{pmatrix}\in\mathbb R^{(m_E+m_I)}$. We choose the instances that can be solved by PDHG (after preconditioning, see below for more details) {\blue to $10^{-10}$ accuracy of KKT residual (see Progress metric for formal definition)} within $3\times 10^5$ iterations, and there are 50 such instances. {\blue Among all 50 instances, the number of variables ranges from 86 to 3648, the number of constraints ranges from 29 to 6972, and the number of nonzeros in the constraint matrices ranges from 376 to 22584.} Notice that there is a difference between \eqref{eq:practical-lp-form} and the standard form~\eqref{eq:standardform} we present in the paper due to the existence of the inequality constraints. Indeed, all of our theory can be generalized to this form without loss of generality. We choose to present the results for the standard form LP \eqref{eq:standardform} because that leads to theoretical results in relatively clean forms.

{\bf Preprocessing.} As a first-order method, PDHG suffers from slow convergence on many instances from \texttt{MIPLIB} due to the ill-conditioning nature of the instances. To overcome this issue, we leverage the diagonal preconditioning heuristic developed in PDLP~\cite{applegate2021practical}.  More specifically, we rescale the constraint matrix $A$ to $\tilde A=D_1AD_2$ with positive definite diagonal matrices $D_1$ and $D_2$. Both vectors $b$ and $c$ are correspondingly rescaled as $\tilde b=D_1b$ and $\tilde c=D_2c$. The matrices $D_1$ and $D_2$ are obtained by running 10 steps of Ruiz scaling \cite{ruiz2001scaling} followed by an $l_2$ Pock-Chambolle scaling \cite{pock2011diagonal} on the constraint matrix $A$. More details about this preconditioning scheme can be found in~\cite{applegate2021practical}.


{\bf Progress metric.} We use KKT residual of \eqref{eq:practical-lp}, i.e., a combination of primal infeasibility, dual infeasibility and primal-dual gap, to measure the performance of current iterates.  More formally, the KKT residual of \eqref{eq:practical-lp-form} is given by 
\begin{equation*}
    \mathrm{KKT}(x,y)=\left\Vert \begin{pmatrix}
        A_{E}x-b_E \\ [A_{I}x-b_{I}]^+ \\ [-x]^+  \\ [A^Ty-c]^+ \\ [y_I]^+ \\  [c^Tx-b^Ty]^+
    \end{pmatrix}  \right\Vert_2 \ .
\end{equation*}

{\bf Computing identification time and $\delta$.} In the experiments, we run PDHG to $10^{-8}$ accuracy to obtain the converging optimal solution $z^*=(x^*,y^*)$. With $z^*$, we can then identify the non-basic solution set $N$, non-degenerate basic solution set $B_1$, and the degenerate basic solution set $B_2$ with respect to $z^*$ for the primal and dual variables, respectively, by
$$N^P=\{i\in [n]: c_i-A_i^Ty^* >0\}, B^P_1=\{i\in [n]\backslash N^P:x_i^*>0\}, B^P_2=[n]\backslash N^P\backslash B^P_1 \ ,$$ and $$N^D=\{j\in I: (b-Ax^*)_j>0\}, B^D_1=\{j\in I\backslash N^D:y_j^*<0 \}, B^D_2=I\backslash N^D\backslash B^D_1 \ .$$

Then we go backward from the last iteration to find the first moment when either $N$ or $B_1$ changes. This iteration is referred to as the empirical identification moment.  The non-degeneracy metric $\delta$ for \eqref{eq:practical-lp} can be computed via
\begin{equation}\label{eq:delta-general}
    \delta = \min\left\{ \min_{i\in N^P}\frac{c_i-A_i^Ty^*}{\|A\|_2}, \min_{i\in B^P_1} x^*_i, \min_{j \in N^D} \frac{(b-Ax^*)_j}{\|A\|_2}, \min_{j\in B_1^D} -y_j^* \right\}\ .
\end{equation}

{\bf Results.} Figure \ref{fig:iteration} presents KKT residual versus  PDHG iterations on six representative instances. The orange vertical line in Figure \ref{fig:iteration} represents the empirical identification moment, i.e., subsequently the active set will be fixed and identical to the limit optimal point. As illustrated in Figure \ref{fig:iteration}, there is a phase transition happening around the identification moment (orange line): PDHG identifies active variables in Stage I with a sublinear rate; PDHG has a fast linear convergence in Stage II after identification. 

\begin{figure}[ht!]
	\centering
	\begin{tabular}{l c c c c}
		\hspace{-2Cm}
		& \includegraphics[width=0.3\textwidth]{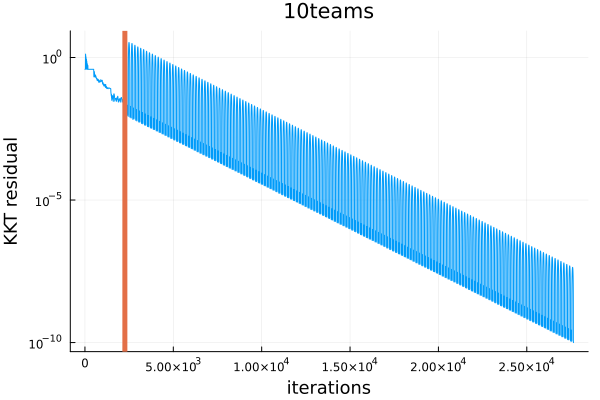}
		\hspace{-0.5Cm}
		& \includegraphics[width=0.3\textwidth]{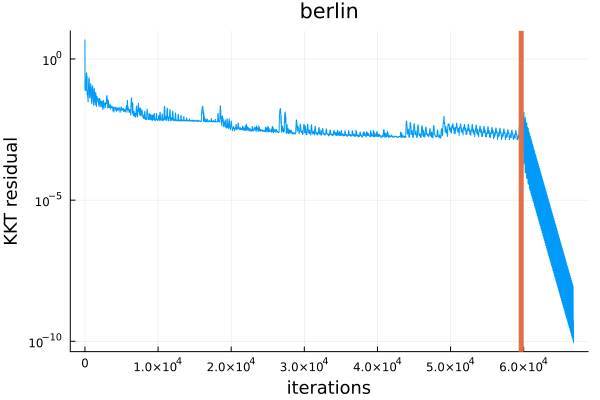}
        \hspace{-0.5Cm}
        & \includegraphics[width=0.3\textwidth]{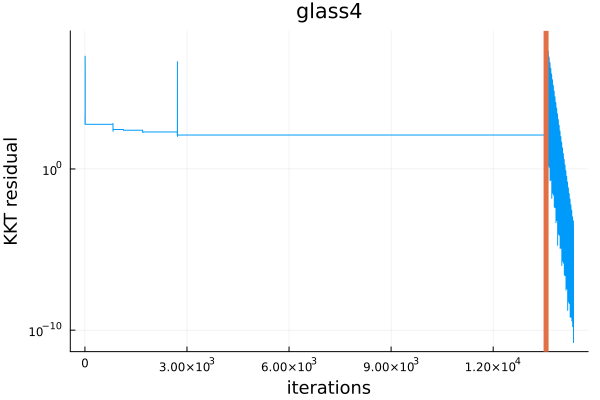}\\
        \hspace{-2Cm}
        & \includegraphics[width=0.3\textwidth]{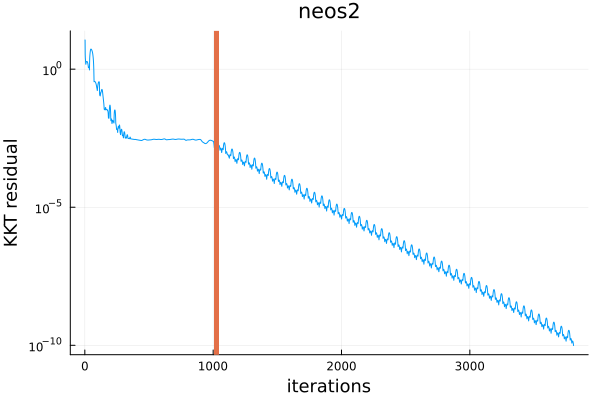}
		\hspace{-0.5Cm}
		& \includegraphics[width=0.3\textwidth]{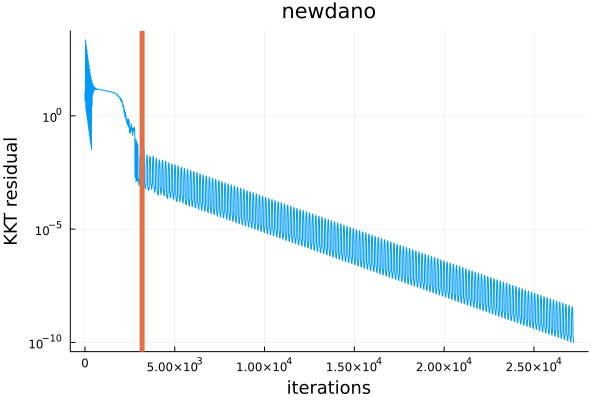}
        \hspace{-0.5Cm}
        & \includegraphics[width=0.3\textwidth]{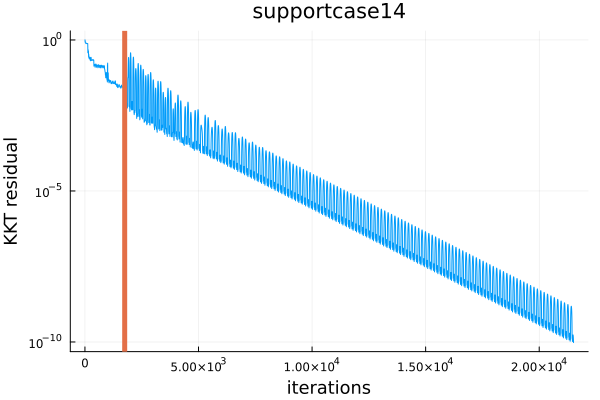}\\
		\hspace{-1Cm}
	\end{tabular}
	\caption{Plots showing the KKT residual in $\log$ scale versus number of iterations of PDHG on six representative LP instances from \texttt{MIPLIB}.}
	\label{fig:iteration}
\end{figure}

\begin{figure}[ht!]
	\centering
	\includegraphics[scale=0.4]{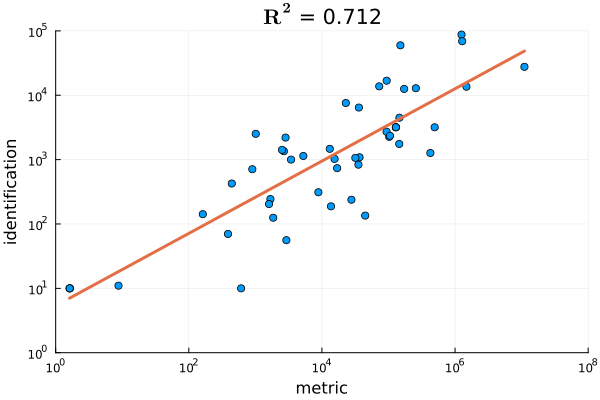}
	\caption{Plot showing the iteration number at identification in $\log$ scale versus non-degeneracy metric $R/\delta$ in $\log$ scale on the 50 \texttt{MIPLIB} instances.}
	\label{fig:scatterplot}
\end{figure}
Theorem \ref{thm:thm-identification} shows a bound on the identification time with the non-degeneracy parameter $\delta$. Figure \ref{fig:scatterplot} presents the scatter plot of the metric $R/\delta$ versus the empirical iteration number for identification, where $R=2\|z^0-z^*\|_2+2\|z^*\|_2+1$ and $\delta$ is defined as \eqref{eq:delta-general}. Each point in Figure \ref{fig:scatterplot} represents an LP instance from our dataset. We observe from Figure \ref{fig:scatterplot} that there is a strong correlation between the empirical identification time and the metric $R/\delta$. Indeed, when running a linear regression model, the $R^2$ is $0.712$, {which means that $71.2\%$ of the variability in the observed identification complexity is captured by this single quantity.}


\begin{figure}[ht!]
	\centering
	\begin{tabular}{c c c c}
		\hspace{-1Cm}
		& \includegraphics[width=0.35\textwidth]{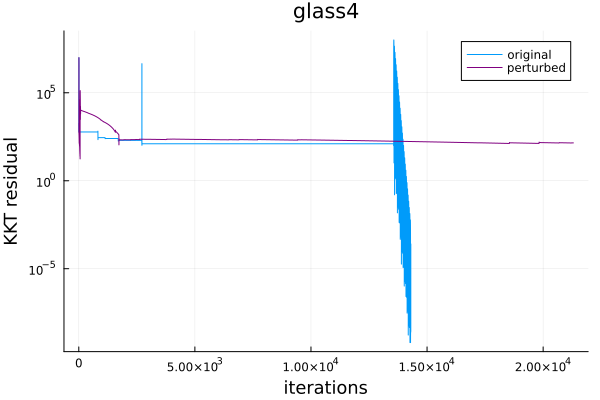}
		\hspace{-0.5Cm}
		& \includegraphics[width=0.35\textwidth]{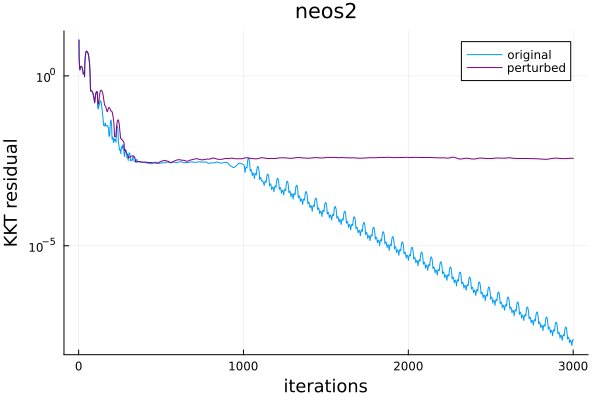}
        \hspace{-0.5Cm}
        & \includegraphics[width=0.35\textwidth]{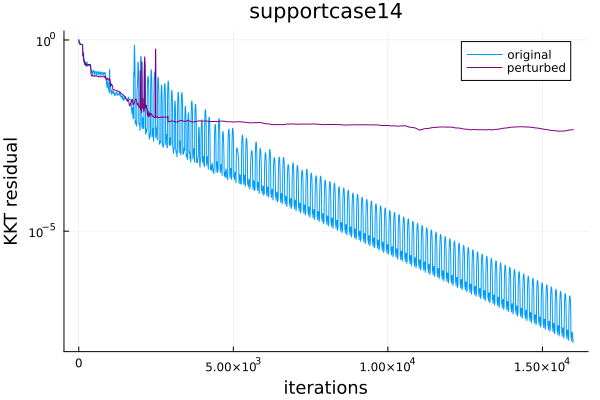}
	\end{tabular}
	\caption{Comparison of original and perturbed LP instances.}
	\label{fig:iteration-perturb}
\end{figure}

An interesting implication of our result is that degeneracy itself does not slow down the convergence, but the near-degeneracy of the non-degenerate part does. Figure \ref{fig:iteration-perturb} demonstrates this theoretical implication by preliminary numerical experiments. We choose three representative LP instances from the dataset. Then, we slightly perturb the instances by adding a small Gaussian noise in $A$, $b$, and $c$, then run PDHG on both the original and the perturbed instances. {\blue We use Gurobi to verify the feasibility of the three perturbed problems.} Figure \ref{fig:iteration-perturb} plots the KKT residual and the number of PDHG iterations for the original LP and the perturbed LP. For each instance in Figure \ref{fig:iteration-perturb}, the blue line represents the performance of PDHG on the original LP, and the purple line is on the perturbed LP.  While the original LP and the perturbed LP are almost identical (because the perturbation is tiny), PDHG converges much faster on the original LP than the perturbed LP. Indeed, almost all practical LP, including the three instances in Figure \ref{fig:iteration-perturb}, are degenerate. Adding the tiny perturbation makes the problem non-degenerate. {\blue This showcases that degeneracy itself does not slow down the convergence of PDHG, but near-degeneracy can make the instance much harder to solve, which verifies our theory.}


In summary, the numerical experiment verifies our geometric understandings of PDHG on LPs: (1) the convergence behavior of PDHG for LP have two distinct stages, that is, Stage I for active variable identification and Stage II for local linear convergence; (2) the non-degeneracy metric $\delta$ plays a crucial role in the length of the first stage; (3) degeneracy itself does not deteriorate the convergence of first-order methods, but near-to-degeneracy does.



\section{Conclusion and future directions}\label{sec:conclusion}
In conclusion, this paper aims to bridge the gap between the practical success and loose complexity bound of PDHG for LP, and to identify geometric quantities of LP that drive the performance of the algorithm. To achieve this, we show that the behaviors of PDHG for LP have two stages: in the first sublinear stage, the algorithm identifies the non-basic and non-degenerate basic set (i.e., the set $N$ and $B_1$) and we present the complexity result of identification in terms of the near-degeneracy metric $\delta$; in the second linear stage, the algorithm effectively solves a homogeneous linear inequality system, and the linear convergence rate is driven by the local sharpness parameter of the system. Compared to existing complexity results of PDHG on LP, our results do not depend on the global Hoffman's constant of the KKT system, which is known to be too loose in the literature. Such two-stage behavior is also tied with the concept of partial smoothness in the non-smooth optimization literature. We introduce a new framework of complexity analysis without assuming the ``irrepresentable'' non-degeneracy condition.

We end the paper by presenting a few interesting open questions for further investigation:
\begin{itemize}
    \item \textbf{Extensions to non-smooth optimization.} The partial smoothness literature in non-smooth optimization always assumes the non-degeneracy condition. {Unfortunately, the iterates of first-order methods usually converge to a degenerate solution, thus violating this condition.
    } A typical argument for avoiding this issue in theory is by adding a small perturbation to the original problem. However, as shown in our numerical experiments (Figure \ref{fig:iteration-perturb}), adding the small perturbation can significantly slow down the convergence of the algorithm. Instead, we show in this paper that degeneracy itself does not slow down the convergence in the context of LP. How to extend this result to general non-smooth optimization can be an interesting future direction.
    \item \textbf{Extensions to other primal-dual algorithms.} We believe the two-stage behaviors are not limited to PDHG. Indeed, most of the theoretical results can be extended to ADMM. This should provide new theoretical understandings for ADMM-based solvers, such as OSQP~\cite{stellato2020osqp} and SCS~\cite{o2016conic}.
    \item \textbf{Extensions to the restarted algorithm.} PDLP utilizes a new technique, restarting, to obtain the optimal linear convergence rate with global sharpness condition. We believe our theoretical results can be extended to the restarted algorithms, and we leave it as future work.
    \item \textbf{Extensions to infeasible problems.} In the paper, we consider the case when the LP instance is feasible and bounded. Such two-stage behaviors are also numerically observed in infeasible problems. How to theoretically analyze it without the non-degeneracy condition can be another interesting question.
\end{itemize}

\bibliographystyle{amsplain}
\bibliography{ref-papers}

\appendix

\section{Proof of Theorem \ref{thm:thm-linear-z}}\label{app:proof_thm_1}
In this section, we provide a proof to Theorem \ref{thm:thm-linear-z}. We begin with a lemma that bridges the distance to converging point with the distance to optimality.
\begin{lem}\label{lem:lem-dist-opt}
Consider the iterations $\{z^k\}_{k=0}^\infty$ of PDHG. Suppose the step-size $s<\frac{1}{\Vert A\Vert}$. Then it holds for any $k\geq 0$ that
\begin{align*}
    \begin{split}
        \Vert z^k-z^* \Vert_{P_s} \leq 2\mathrm{dist}_{P_s}(z^k,\mathcal Z^*) \ ,
    \end{split}
\end{align*}
where $z^*=\lim_{k\rightarrow\infty} z^k$.
\end{lem}
\begin{proof}
By non-expansiveness of the iterates, we have
    \begin{align*}
        \begin{split}
            \Vert z^n-\mathrm{proj}_{\mathcal Z^*}(z^k) \Vert_{P_s} \leq \Vert z^k-\mathrm{proj}_{\mathcal Z^*}(z^k) \Vert_{P_s}=\mathrm{dist}_{P_s}(z^k,\mathcal Z^*) \ .
        \end{split}
    \end{align*}
    Let $n\rightarrow\infty$. It holds that
    \begin{align*}
        \begin{split}
            \Vert z^*-\mathrm{proj}_{\mathcal Z^*}(z^k) \Vert_{P_s} \leq \mathrm{dist}_{P_s}(z^k,\mathcal Z^*) \ .
        \end{split}
    \end{align*}
    Thus we finish the proof by triangle inequality
    \begin{align*}
        \begin{split}
            \Vert z^k-z^* \Vert_{P_s} \leq \Vert z^k-\mathrm{proj}_{\mathcal Z^*}(z^k) \Vert_{P_s} + \Vert \mathrm{proj}_{\mathcal Z^*}(z^k)-z^* \Vert_{P_s} \leq 2\mathrm{dist}_{P_s}(z^k,\mathcal Z^*) \ .
        \end{split}
    \end{align*}
\end{proof}

It turns out that distance to optimality has linear convergence under sharpness \eqref{eq:sharp}.
\begin{lem}\label{lem:linear-converge}
    Consider the iterations $\{z_k\}_{k=0}^{\infty}$ of PDHG (Algorithm \ref{alg:pdhg}) to solve a convex-concave primal-dual problem. Suppose the step-size $s\leq  \frac{1}{2\Vert A \Vert}$, and the primal-dual problem satisfies sharpness condition \eqref{eq:sharp} on a set $\mathcal S$ that contains $\{z_k\}_{k=0}^{\infty}$. Then, it holds for any iteration $k\ge 0$ that
    \begin{equation*}
        \mathrm{dist}^2_{P_s}(z^k,\mathcal Z^*)\leq \exp\pran{1-\frac{k}{\left\lceil 4e/(s\alpha)^2\right\rceil}}\mathrm{dist}^2_{P_s}(z^0,\mathcal Z^*)
    \end{equation*}
\end{lem}
\begin{proof}

By \cite[Proposition 1]{lu2021linear} and Lemma \ref{lem:lem-norm}, we know the $\alpha$-sharpness condition implies that the following equation holds for any $k\geq 0$
\begin{equation}
    \mathrm{dist}_{P_s}(z^k,\mathcal Z^*)\leq \frac{2}{s\alpha}\mathrm{dist}_{P_s^{-1}}(0,\mathcal F(z^k))
\end{equation}

Then the proof is similar to \cite[Theorem 2]{lu2022infimal}. For any iteration $k\geq 1$, suppose $c\left\lceil \frac{4e}{s^2\alpha^2}\right\rceil \leq k < (c+1)\left\lceil \frac{4e}{s^2\alpha^2}\right\rceil$ for a non-negative integer $c$.
\begin{align*}
    \begin{split}
        \mathrm{dist}^2_{P_s}(z^k,\mathcal Z^*) & \ \leq \frac{4}{s^2\alpha^2}\mathrm{dist}^2_{P_s^{-1}}(0,\mathcal F(z^k)) \\
        & \ \leq \frac{4}{s^2\alpha^2}\frac{1}{\lceil 4e/(s^2\alpha^2) \rceil}\mathrm{dist}^2_{P_s}(z^{k-\lceil 4e/(s\alpha)^2 \rceil},\mathcal Z^*) \\
        & \ \leq \frac{1}{e}\mathrm{dist}^2_{P_s}(z^{k-\lceil 4e/(s\alpha)^2 \rceil},\mathcal Z^*) \\
        & \ \leq ... \\
        & \ \leq \pran{\frac 1e}^c \mathrm{dist}^2_{P_s}(z^{k-c\lceil 4e/(s\alpha)^2 \rceil},\mathcal Z^*) \\
        & \ \leq \pran{\frac 1e}^c \mathrm{dist}^2_{P_s}(z^{0},\mathcal Z^*) \\
        & \ \leq \exp\pran{1-\frac{k}{\left\lceil 4e/(s\alpha)^2\right\rceil}}\mathrm{dist}^2_{P_s}(z^{0},\mathcal Z^*)
    \end{split}
\end{align*}
where the second inequality follows from \cite[Theorem 1]{lu2022infimal}.
\end{proof}

\begin{proof}[Proof of Theorem \ref{thm:thm-linear-z}]
(1) The result follows from \cite[Theorem 1]{lu2022infimal}.

(2) From Lemma \ref{lem:linear-converge}, we have
\begin{equation}\label{eq:eq-linear-IDS}
    \mathrm{dist}_{P_s}(z^k,\mathcal Z^*)\leq \exp\pran{\frac 12-\frac{k}{2\lceil 4e/(s\alpha)^2\rceil}}\mathrm{dist}_{P_s}(z^0,\mathcal Z^*) \ .
\end{equation}

Therefore we reach
\begin{align*}
    \begin{split}
         \Vert z^k-z^* \Vert_{P_s} & \ \leq 2\mathrm{dist}_{P_s}(z^k,\mathcal Z^*) \\
         & \ \leq 2\exp\pran{\frac 12-\frac{k}{2\left\lceil 4e/(s\alpha)^2\right\rceil}}\mathrm{dist}_{P_s}(z^0,\mathcal Z^*)\\
         & \ \leq 4\exp\pran{-\frac{k}{2\left\lceil 4e/(s\alpha)^2\right\rceil}}\mathrm{dist}_{P_s}(z^0,\mathcal Z^*)\ ,
    \end{split}
\end{align*}
where the first inequality is exactly Lemma \ref{lem:lem-dist-opt} and the second one utilizes \eqref{eq:eq-linear-IDS}.

\end{proof}


    

\section{Examples to illustrate the looseness of global sharpness constant}\label{app:example}
In this section, we present a series of simple LP instances which exhibits very bad global sharpness behaviors (so as the Hoffman constant), whereas the three quantities we use in the analysis, $\delta$, $\alpha_{L_1}$ and $\alpha_{L_2}$, are of reasonable sizes. This indicates that the previous complexity result based on Hoffman constant argument~\cite{applegate2023faster} can be arbitrarily loose, and may not be able to explain the behaviors of the algorithm.

\begin{figure}[h!]
\centering
  \hspace{-1cm}
  \begin{tikzpicture}[scale=0.7,every node/.style={scale=1.5}]
        \begin{axis}[xlabel={},ylabel={}, ytick={0}, xtick={0}, yticklabels={}, xticklabels={}, extra x ticks={0, 1, 5},extra x tick labels={$0$, $1$, $1+1/\kappa$},  extra y ticks={-0.25, 0, 1}, extra y tick labels={$-\kappa$, $0$, $1$},legend pos=outer north east, legend cell align=left, xmin=-0.25, xmax=5.5, ymin=-0.5, ymax=1.25, y=2.8cm,  x=2.8cm,]
        \addplot[forget plot,color=black,line width=1.5pt, fill=gray!40] coordinates {(0,-0.25) (0,0) (0,1) (5,1) (0,-0.25)};
        \addplot[-Stealth, color=black,line width=1.5pt] coordinates {(5.25,0.8) (5.25,1.2)};
        \end{axis}
        \end{tikzpicture}
  \caption{Illustration of the LP instance \eqref{eq:hoffman-bad}}
  \label{fig:hoffman-bad}
\end{figure}
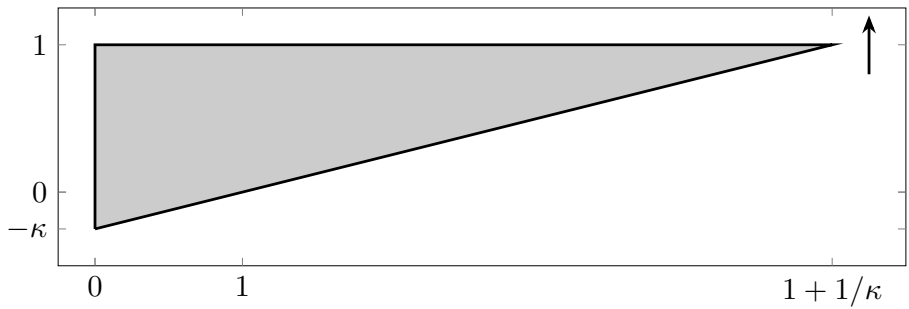

Consider a two dimension dual LP parameterized by $\kappa>0$:
\begin{align}\label{eq:hoffman-bad}
    \begin{split}
        & \ \max \; y_2 \\
        & \ \;\; \mathrm{s.t.}\;\; y_1\geq 0,\; y_2\leq 1 \\ 
        & \ \quad\quad\;\; {\kappa}y_1- y_2 \leq \kappa \ ,
    \end{split}
\end{align}
where the primal form \eqref{eq:standardform} is given by
\begin{equation}\label{eq:hoffman-data}
    A=\begin{pmatrix}
        0 & -1 & \kappa \\ 1 & 0 &-1
    \end{pmatrix},\;\qquad b=\begin{pmatrix}
        0 \\ 1
    \end{pmatrix},\;\qquad c=\begin{pmatrix}
        1 \\ 0\\ \kappa
    \end{pmatrix}.
\end{equation}

The feasible region of \eqref{eq:hoffman-bad} is plotted in Figure \ref{fig:hoffman-bad}. It is easy to check that there are multiple dual optimal solutions given by $y^*=(y_1,1)$ with $0\leq y_1\leq 1+\frac{1}{\kappa}$, and the optimal primal solution is unique given by $x^*=(1,0,0)$. Denote $\alpha$ the sharpness constant (i.e., the reciprocal of the Hoffman constant) of the KKT system
\begin{equation*}
    Ax=b,\; A^Ty\leq c,\; x\geq 0,\; \frac 1R (c^Tx-b^Ty)\leq 0 \ ,
\end{equation*}
where $(A,b,c)$ is given in \eqref{eq:hoffman-data} and $R$ is essentially the distance between initial point to origin. This is the term that appears in the complexity of previous results~\cite{applegate2023faster}. Table \ref{tab:hoffman-bad} presents bounds on $\alpha$, $\alpha_{L_{1}}$, $\alpha_{L_{2}}$ and $\delta$ for the LP when $\kappa$ is set to $10^{-10}$. The value of $\alpha$ can be obtained by noticing that $\alpha\le \kappa$ (see Lemma \ref{lem:kappaalpha} for a proof). The bound of $\alpha_{L_1}$ and $\alpha_{L_2}$ are obtained from the algorithm presented in \cite{pena2023easily}. The bound on $\delta$ can be computed via Lemma \ref{prop:hoffman-bad-delta}. As we can see, for this specific example, the value of $\alpha$ is tiny, while the value of $\alpha_{L_{1}}$, $\alpha_{L_{2}}$ and $\delta$ are of reasonable sizes. 

\begin{table}[h]
\centering
\begin{tabular}{|c|c|c|c|c|}
\toprule
$\kappa$ & $\alpha$     & $\alpha_{L_1}$ & $\alpha_{L_2}$ & $\delta$ \\ \midrule
 $=10^{-10}$ & $\leq 10^{-10}$   &  $\geq 0.00821$   &  $\geq 0.0276$ & $\geq 0.363$ \\ \bottomrule
\end{tabular}

\caption{Comparison of global sharpness and metrics proposed for small $\kappa$}
\label{tab:hoffman-bad}
\end{table}

\begin{lem}\label{lem:kappaalpha}
$\alpha\leq {\kappa}$.
\end{lem}
\begin{proof}
    By definition of sharpness constant, it holds for any $z$ that
    \begin{equation*}
        \alpha\mathrm{dist}(z,\mathcal Z^*)\leq \left\Vert \begin{pmatrix}
        Ax-b \\ [A^Ty-c]^+ \\ \frac 1R [c^Tx-b^Ty]^+
    \end{pmatrix}  \right\Vert_2 \ ,
    \end{equation*}
    where $\mathcal Z^*=\{(1,0,0,y_1,1)\;|\; 0\leq y_1\leq 1+\frac{1}{\kappa}\}$.   
    Then for any $\zeta>0$, consider $z=(1,0,0,\frac{1+\kappa}{\kappa}+\zeta,1)$ and we have
    \begin{equation*}
        \alpha^2\zeta^2\leq \pran{\kappa\pran{\frac{1+\kappa}{\kappa}+\zeta}-1-\kappa}^2 \ , 
    \end{equation*}
    which implies
    \begin{equation*}
        \alpha^2\leq \kappa^2 \ .
    \end{equation*}
\end{proof}
\begin{lem}\label{prop:hoffman-bad-delta}
Suppose $\kappa\leq 0.1$ and the initial point $z^0$ satisfies $\|z^0-\tilde z\|_2\leq 1.5$, where $\tilde z=(1,0,0,4,1)\in \mathcal Z^*$. Then it holds that
\begin{equation*}
        \delta \geq \frac{4}{11} \ .
    \end{equation*}
\end{lem}
\begin{proof}
From the non-expansiveness of PDHG iterates and Lemma \ref{lem:lem-norm}, we have
\begin{equation*}
    \|z^k-\tilde z\|_2\leq \sqrt{2s}\|z^k-\tilde z\|_{P_s}\leq \sqrt{2s}\|z^0-\tilde z\|_{P_s} \leq 2\|z^0-\tilde z\|_2\leq 3 \ .
\end{equation*}
Denote $z^*=(x^*,y^*)$ the converging optimal point of $\{z^k\}_{k=0}^\infty$. Then it holds that
\begin{equation*}
    0<1\leq y_1^*\leq 7<1+\frac{1}{\kappa} \ .
\end{equation*}
Thus
\begin{equation*}
    N=\{2,3\},\; B_1=\{1\},\;B_2=\O \ ,
\end{equation*}
and
\begin{equation*}
    c_2-A_2^Ty^* \geq 1,\; c_3-A_3^Ty^*\geq \frac{1+\frac{1}{\kappa}-7}{1+\frac{1}{\kappa}}\geq \frac{4}{11},\; x_1^* = 1 \ ,
\end{equation*}
where $\frac{1+\frac{1}{\kappa}-7}{1+\frac{1}{\kappa}}\geq \frac{4}{11}$ by $\kappa\leq 0.1$. Note that $\|A\|_2\leq \|A\|_F\leq 2$ and we reach at
\begin{equation*}
    \delta= \min\left\{\min_{i\in N} \frac{c_i-A_i^Ty^*}{\|A\|_2},\;\min_{i\in B_1} x_i^*\right\} \geq \frac{4}{22}=\frac{2}{11}\ .
\end{equation*}
\end{proof}

\end{document}